\documentclass[a4paper,12pt]{amsart}
\usepackage[utf8]{inputenc}
\usepackage[english]{babel}
\usepackage{amsmath}  
\usepackage{amssymb}     
\usepackage{amsthm}  
\usepackage{bbm}    

\usepackage{bm}
\usepackage[x11names,rgb,table]{xcolor}
\usepackage[naturalnames=false]{hyperref}
\hypersetup{
  colorlinks = true,          
  urlcolor = DeepSkyBlue3, linkcolor = Chartreuse4, citecolor = DarkOrange2
}
\usepackage{accents}
\usepackage{mathrsfs}
\usepackage{mathtools}
\usepackage{fixmath}
\usepackage{fullpage}
\usepackage{paralist}
\usepackage{enumerate}
\usepackage{stmaryrd}

\usepackage{booktabs}
\usepackage{array}
\usepackage{csquotes}
\usepackage{xspace}

\usepackage{pgf,tikz}
\usetikzlibrary{arrows}
\usetikzlibrary{patterns}
\usetikzlibrary[positioning]
\usetikzlibrary[fit]
\usetikzlibrary{shapes.geometric}
\usetikzlibrary{calc}

\usepackage{pgfplots}
\pgfplotsset{compat=1.17}

\usepackage[boxed,vlined,oldcommands]{algorithm2e}
\SetKwInput{Input}{input}
\SetKwInput{Output}{output}
\SetKw{Break}{break}
\SetArgSty{rm}
\SetFuncSty{tt}
\setalcapskip{1ex}

\newcommand\EE{{\mathbb E}}
\newcommand\NN{{\mathbb N}}

\newcommand\RR{{\mathbb R}}
\newcommand\ZZ{{\mathbb Z}}
\newcommand\cI{{\mathcal I}}
\newcommand\cP{{\mathcal P}}

\newcommand\cS{{\mathcal S}}
\newcommand\cT{{\mathcal T}}

\newcommand\0{{\mathbf 0}}
\newcommand\1{{\mathbf 1}}

\newcommand\SetOf[2]{\left\{#1 \mid #2\right\}}
\newcommand\smallSetOf[2]{\{#1 \mid #2\}}

\newcommand\biggSetOf[2]{\biggl\{#1 \biggm| #2\biggr\}}
\newcommand\BiggSetOf[2]{\Biggl\{#1 \Biggm| #2\Biggr\}}
\newcommand\card[1]{\# #1} 
\newcommand\transpose[1]{{#1}^{\top}}
\newcommand\invtranspose[1]{{#1}^{-\top}}

\DeclareMathOperator\conv{conv}
\DeclareMathOperator\vol{vol}
\DeclareMathOperator\Sym{Sym}
\DeclareMathOperator\argmax{arg\,max}
\DeclareMathOperator\Rev{Rev}

\newcommand\doi[1]{\href{http://dx.doi.org/#1}{\texttt{doi:#1}}}

\newtheorem{theorem}{Theorem}
\newtheorem{proposition}[theorem]{Proposition}
\newtheorem{corollary}[theorem]{Corollary}
\newtheorem{lemma}[theorem]{Lemma}
\theoremstyle{remark}
\newtheorem{remark}[theorem]{Remark}

\newtheorem{example}[theorem]{Example}

\newtheorem{question}[theorem]{Question}

\newcommand\polymake{\texttt{polymake}\xspace}

\newcommand\HC{\texttt{Homotopy\-Continuation.jl}\xspace}

\setdefaultitem{$\triangleright$}{}{}{}

\title{Generalized permutahedra and optimal auctions}

\author{Michael Joswig \and Max Klimm \and Sylvain Spitz} 

\address[Michael Joswig]{
  Technische Universität Berlin,
  Chair of Discrete Mathematics/Geometry; \\
  Max-Planck Institute for Mathematics in the Sciences, Leipzig \\
  \texttt{joswig@math.tu-berlin.de}
}

\address[Max Klimm \and Sylvain Spitz]{
  Technische Universität Berlin,
  Discrete Optimization\\
  \texttt{$\{$klimm,spitz$\}$@math.tu-berlin.de}	
}

\thanks{%
  Support by the Deutsche Forschungsgemeinschaft (DFG, German Research Foundation) under Germany's Excellence Strategy - The Berlin Mathematics Research Center MATH$^+$ (EXC-2046/1, project ID 390685689) gratefully acknowledged.
  M.~Joswig has further been supported by \enquote{Symbolic Tools in Mathematics and their Application} (TRR 195, project-ID 286237555); \enquote{Facets of Complexity} (GRK 2434, project-ID 385256563).}

\begin{document}

\begin{abstract}
  We study a family of convex polytopes, called SIM-bodies, which were introduced by Giannakopoulos and Koutsoupias (2018) to analyze so-called Straight-Jacket Auctions.
  First, we show that the SIM-bodies belong to the class of generalized permutahedra.
  Second, we prove an optimality result for the Straight-Jacket Auctions among certain deterministic auctions.
  Third, we employ computer algebra methods and mathematical software to explicitly determine optimal prices and revenues.
\end{abstract}

\maketitle

\section{Introduction}
The design of auctions that maximize the revenue of a seller is a central question of economic theory. 
In the basic model, there is a single item and a set of buyers interested in obtaining the item from the seller.
Every buyer individually holds a private value, which is the maximal price that they are willing to spend in order to receive the item.
The exact valuation of a buyer is unknown to both the seller and the other buyers, but it is drawn from a probability distribution that is common knowledge.
In this setting, a \emph{mechanism} (or auction) is a procedure that governs how the seller elicits information from the buyers regarding their private valuations and decides to whom to sell the item at which price.
The seller seeks to maximize their expected revenue, i.e., the price received from the successful buyer.
The mathematical properties of auctions in these settings are at the heart of two Nobel Memorial Prizes in Economics. Vickrey~\cite{Vickrey1961} received the price in 1996 for showing that many popular auction formats give the same revenue to the seller, and Myerson~\cite{MR0618964} received the prize in 2007 for a precise characterization of the revenue-maximizing auction. 

While auctions for selling a single item are reasonably well understood, the case where multiple items are for sale is much more challenging due to the underlying combinatorics of the problem.
The main difficulty is to decide which subset of items (or randomizations thereof) to sell to the buyer.
It has been observed that even in the simplest case of a single buyer whose valuation for the items is additive, the optimal auction may be randomized \cite{DaskalakisDT14,Thanassoulis04} and may even require randomization among an uncountable set of allocations~\cite{DaskalakisDT17}.
Optimal auctions are also challenging from a computational point of view since computing the optimal auction is $\mathsf{NP}$-hard for distributions of support three \cite{ChenDPSY18} and $\mathsf{\#P}$-hard in general \cite{DaskalakisDT14}. There are two main ways to deal with these complexities.
The first one is to study suboptimal auctions; this approach avoids the combinatorial complexity by selling only all items or only a single item; see, e.g., \cite{BabaioffILW20,HartN17,LiY13}.
The second one considers specific distributions that allow to study the combinatorial structure of the problem.
Giannakopoulos and Koutsoupias \cite{GK+duality:2018} examined the case where the valuation for each item is drawn from the uniform distribution and devise a deterministic auction, called \emph{Straight-Jacket Auction (SJA)}. 
As their main result, Giannakopoulos and Koutsoupias showed that SJA is optimal for up to to six items \cite[Theorem 4.8]{GK+duality:2018}.

Giannakopoulos and Koutsoupias recognized that crucial properties of SJA are expressed in terms of certain convex polytopes and their volumes; in \cite{GK+duality:2018} these polytopes were called \emph{SIM-bodies}; yet they previously occurred as \enquote{$Q$-polytopes} in work of Doker~\cite[\S2.4]{Doker:phd}.
The SIM-bodies form our point of departure, and we derive crucial structural information.
Most notably, these turn out to be generalized permutahedra; this is our Theorem~\ref{thm:generalized}.
Generalized permutahedra were introduced by Postnikov \cite{Postnikov:2009}, and they received considerable attention in the algebraic combinatorics community and beyond.
Moreover, generalized permutahedra are most tightly related to submodular functions, polymatroids and $M$-convexity studied in optimization; see Frank and Murota \cite{FrankMurota:1808.07600} for a survey.
For instance, generalized permutahedra already occur as \enquote{cores of cooperative games} in work of Danilov and Koshevoy \cite[Proposition 5]{DanilovKoshevoy:2000}.
As our main theoretical contribution, in Theorem \ref{thm:uniqueCritPrice}, we prove a general uniqueness result for submodular and deterministic auctions.
This leads us to proving the optimality of SJA for any number of items, but for a restricted choice of auctions, and subject to a technical condition (Corollary~\ref{cor:SJA-optimal}).
The existence of such auctions is far from obvious; in fact, the article \cite{GK+duality:2018} left open if SJA exists for $n\geq 7$ items.
A SIM-body arises as that portion of the domain of the utility of the buyer where no bundle is sold.
Equivalently, SIM-bodies can be described as regions of tropical hypersurfaces induced by the utilities.
In this sense our present work is in line with recent efforts to employ methods from tropical geometry to topics in mechanism design \cite{BaldwinKlemperer,TranYu:2019,CrowellTran:1606.04880,ETC}.

The rest of our paper is devoted to actually computing the SJA-prices and their revenues.
Our method builds on a new parametric volume formula for SIM-bodies, Proposition~\ref{prop:SIM-lawrence}, which we derive from a general algorithm of Lawrence \cite{Lawrence:1991}.
We exploit that this method for computing the volume of a convex polytope is particularly nice for (rational) polytopes whose normal fans are smooth in the sense of toric geometry.
An $n$-dimensional SIM-body depends on $n$ parameters, which are real numbers corresponding to the prices of the $n$ items in the auction;.
The normalized volume is an integral homogeneous polynomial of degree $n$ in these $n$ parameters.
Computing the SJA-prices then amounts to finding real roots of a system of $n$ polynomials equations in $n$ indeterminates and testing submodularity conditions (which yield linear inequalities).
In this way, we establish the existence of SJA for $n\leq 12$ items by computations using \polymake \cite{DMV:polymake} and \HC \cite{HC}.
Furthermore, it turns out that a certain linear substitution of the volume polynomials of the SIM-bodies gives Lorentzian polynomials \cite{Lorentzian}.

\subsection*{Acknowledgments}
We are grateful to Katharina Jochemko for pointing out \cite{DanilovKoshevoy:2000}.


\section{SIM-bodies}
\label{sec:sim}
Our approach to studying auctions is primarily geometric and algorithmic.
To this end it is useful to begin with certain classes of (convex) polytopes.
We refer to \cite{Ziegler:Lectures+on+polytopes} and \cite{Polyhedral+and+Algebraic+Methods} for the basics of polyhedral geometry and related algorithms.

The \emph{permutahedron} $\cP_{n+1}(\alpha_1,\dots,\alpha_{n+1})$ is the convex hull of all $(n+1)!$ points which arise from coordinate permutations of $(\alpha_1,\dots,\alpha_{n+1})$, and the parameters $\alpha_i$ are arbitrary real numbers.
By construction the symmetric group $\Sym(n+1)$ of degree $n+1$ operates on $\cP_{n+1}(\alpha_1,\dots,\alpha_{n+1})$ by linear automorphisms.
Examples include the \emph{regular permutahedron} $\cP_{n+1}(n,n-1,\dots,0)$ and the \emph{hypersimplex} $\Delta(k,n)=\cP_n(1,\dots,1,0,\dots,0)$, with $k$ ones and $n-k$ zeros.
The following is an $H$-description of permutahedra, which is known; see \cite[Proposition 2.5]{Postnikov:2009}, where that result is attributed to Rado~\cite{Rado:1952}. 
\begin{lemma}\label{lem:permutahedron}
  For $\alpha_1 \geq \cdots \geq \alpha_{n+1}$ we have
  \[
    \begin{split}
      \cP_{n+1}(&\alpha_1,\dots,\alpha_{n+1}) \\
      \ &= \ \biggSetOf{ x\in\RR^{n+1} }{ \sum_{i=1}^{n+1} x_i = \bar\alpha \,,\ \sum_{i\in I} x_i \leq \alpha_{1}+\dots + \alpha_{k} \text{ for } I\in\tbinom{[n+1]}{k} \,,\ k\geq 1 } \enspace ,\\
    \end{split}
  \]
  where $\bar\alpha := \sum\alpha_i$.
\end{lemma}
Throughout we abbreviate $[n]=\{1,2,\dots,n\}$, and we write $\tbinom{[n]}{k}$ for the set of $k$-element subsets of $[n]$.
A key construction comes from considering a map $z$ that assigns each nonempty subset $I \subseteq [n+1]$ a real number~$z_I$.
Following Postnikov~\cite[\S6]{Postnikov:2009} this gives rise to the polyhedron
\begin{equation}
  \label{eq:generalized}
  \cP_{n+1}\{z\} \ = \ \biggSetOf{ x\in\RR^{n+1} }{ \sum_{i=1}^{n+1} x_i = z_{[n+1]} \,,\ \sum_{i\in I} x_i \leq z_{[n+1]}-z_{I^c} \text{ for } \emptyset\neq I\subseteq[n+1] } \enspace ,
\end{equation}
where $I^c = [n+1]\setminus I$ is the complement of $I$ in $[n+1]$.
\begin{example}
  For $z_I=\alpha_{n+2-k}+\dots + \alpha_{n+1}$ with $k=\card{I}$ we recover the permutahedron $\cP_{n+1}(\alpha_1,\dots,\alpha_{n+1})=\cP_{n+1}\{z\}$.
\end{example}
Now we pick a fixed weakly descending sequence of $n$ nonnegative real numbers $\alpha_1 \geq \cdots \geq \alpha_n \geq 0$.
Specializing \eqref{eq:generalized} by letting
\begin{equation} \label{eq:z_I}
  z_I \ = \ \begin{cases}
    \sum_{i=1}^{\card{I} - 1} \alpha_{n+1-i} & \text{ if } n+1 \in I \\
    0 & \text{ otherwise} \enspace,
  \end{cases}
\end{equation}
we obtain the polyhedron $\Gamma(\alpha_1,\dots,\alpha_n):=\cP_{n+1}\{z\}$.
Observe that $z_{[n+1]}=\alpha_{1}+\dots + \alpha_{n}=\bar\alpha$.
Moreover, we define $\Lambda(\alpha_1,\dots,\alpha_n)$ as the orthogonal projection of $\Gamma(\alpha_1,\dots,\alpha_n)$ gotten by omitting the last coordinate.
As the polyhedron $\Gamma(\alpha_1,\dots,\alpha_n)$ is contained in the affine subspace $\sum_{i=1}^{n+1} x_i = \bar\alpha$, the orthogonal projection $\Lambda(\alpha_1,\dots,\alpha_n)$ is a polyhedron in $\RR^n$ which is affinely isomorphic.

\begin{proposition}\label{prop:projection}
  For $\alpha_1 \geq \cdots \geq \alpha_n \geq 0$ we have
  \[
    \Lambda(\alpha_1,\dots,\alpha_n) \ = \ \biggSetOf{ x \in \RR_{\geq 0}^n }{  \sum_{i \in I} x_i \leq \alpha_{1}+\dots + \alpha_k \text{ for } I\in\tbinom{[n]}{k} \text{ and } k\geq 1 } \enspace ,
  \]
\end{proposition}
\begin{proof}
  Let $\pi:\RR^{n+1}\to\RR^n$ be the projection which omits the last coordinate, and let $z_I$ be defined for nonempty $I\subseteq[n+1]$ as in \eqref{eq:z_I}.
  Applying one Fourier--Motzkin step yields
  \[
    \pi \bigl( P_{n+1}\{z_I\} \bigr) \ = \ \biggSetOf{ x \in \RR^n }{ z_I \leq \sum_{i \in I} x_i \leq \bar\alpha - z_{I^c} \text{ for } \emptyset\neq I\subseteq[n] } \enspace ,
  \]
  where $I^c = [n+1]\setminus I$.
  This gives the claim.
\end{proof}
Proposition~\ref{prop:projection} says that $\Lambda(\alpha_1,\dots,\alpha_n)$ is a \emph{SIM-body} in the sense of Giannakopoulos and Koutsoupias \cite[Definition~4.6]{GK+duality:2018}.
However, note that our parameters $\alpha_i$ are descending, whereas they are chosen ascending in \cite{GK+duality:2018}.
This deviation is suggested by aligning with Postnikov's notation \cite{Postnikov:2009}.
The following is a $V$-description of the SIM-bodies.

\begin{figure}[th]\label{fig:SIM-bodies}
	\begin{minipage}{.45\textwidth}
	  \begin{tikzpicture}[scale=1.5]
  		\tikzstyle{face}=[gray!40, fill opacity = .5];
  		\tikzstyle{axis}=[thin,gray,-stealth];
  		
	    \draw[axis] (0,0) -- (3,0) node[right] {$x_1$};
	    \draw[axis] (0,0) -- (0,3) node[above] {$x_2$};
	    \draw (0,0)--(2,0)--(2,1)--(1,2)--(0,2)--cycle;
			\fill[face] (0,0)--(2,0)--(2,1)--(1,2)--(0,2)--cycle;
	    
	    \draw[fill=gray] (0,0) circle (1pt);
	    \draw[fill=gray] (2,0) circle (1pt) node[below] {$\alpha_1$};
	    \draw[fill=gray] (0,2) circle (1pt) node[left] {$\alpha_1$};
	    \draw[fill=gray] (2,1) circle (1pt) node[right] {$(\alpha_1, \alpha_2)$};
	    \draw[fill=gray] (1,2) circle (1pt) node[above right] {$(\alpha_2, \alpha_1)$};
	    
	  \end{tikzpicture}
	\end{minipage}
	\begin{minipage}{.45\textwidth}
		\begin{tikzpicture}[scale=1]
  		\tikzstyle{face}=[gray!40, fill opacity = .5];
  		\tikzstyle{axis}=[thin,gray,-stealth];
  		
			\coordinate (O) at (0,0,0);			
			\coordinate (A1) at (3,0,0);
			\coordinate (A2) at (0,3,0);
			\coordinate (A3) at (0,0,3);
			\coordinate (B12) at (3,2,0);
			\coordinate (B21) at (2,3,0);
			\coordinate (B13) at (3,0,2);
			\coordinate (B31) at (2,0,3);
			\coordinate (B23) at (0,3,2);
			\coordinate (B32) at (0,2,3);
			\coordinate (C12) at (3,2,1);
			\coordinate (C21) at (2,3,1);
			\coordinate (C13) at (3,1,2);
			\coordinate (C31) at (2,1,3);
			\coordinate (C23) at (1,3,2);
			\coordinate (C32) at (1,2,3);
			
			\draw[axis] (O)--(4,0,0) node[right] {$x_1$};
			\draw[axis] (O)--(0,4,0) node[above] {$x_2$};
			\draw[axis] (O)--(0,0,4) node[below] {$x_3$};
			\draw (O)--(A1)--(B12)--(C12)--(C13)--(B13)--(A1);
			\draw (O)--(A2)--(B21)--(C21)--(C23)--(B23)--(A2);
			\draw (O)--(A3)--(B32)--(C32)--(C31)--(B31)--(A3);
			\draw (B12)--(B21);
			\draw (B13)--(B31);
			\draw (B23)--(B32);
			\draw (C12)--(C21);
			\draw (C13)--(C31);
			\draw (C23)--(C32);

      \draw[fill=gray] (O) circle (1.5pt);			
			
			\fill[face] (A1)--(B12)--(C12)--(C13)--(B13)--cycle;
			\fill[face] (A2)--(B21)--(C21)--(C23)--(B23)--cycle;
			\fill[face] (A3)--(B32)--(C32)--(C31)--(B31)--cycle;
			\fill[face] (B12)--(B21)--(C21)--(C12)--cycle;
			\fill[face] (B13)--(B31)--(C31)--(C13)--cycle;
			\fill[face] (B32)--(B23)--(C23)--(C32)--cycle;
			\fill[face] (C12)--(C21)--(C23)--(C32)--(C31)--(C13)--cycle;
			
			\node at (A1) [below=3pt] {$\alpha_1$};
			\node at (A2) [left=3pt] {$\alpha_1$};
			\node at (A3) [left=3pt] {$\alpha_1$};
			\node at (C12) [left=3pt] {$v$};
			
			\foreach \v in {A1,A2,A3,B12,B21,B13,B31,B23,B32,C12,C21,C13,C31,C23,C32} \draw[fill=gray] (\v) circle (1.5pt);
		\end{tikzpicture}
  \end{minipage}
  \caption{SIM-bodies $\Lambda(\alpha_1, \alpha_2)$ and $\Lambda(\alpha_1, \alpha_2, \alpha_3)$.
    The vertex labeled $v$ in the right picture has the coordinates $(\alpha_1, \alpha_2, \alpha_3)$. 
    Note that the faces of the SIM-body, which are normal to the coordinate directions, are also SIM-bodies.
    This property of SIM-bodies and more were shown in \cite[Lemma 6.1]{GK+duality:2018}.} 
\end{figure}
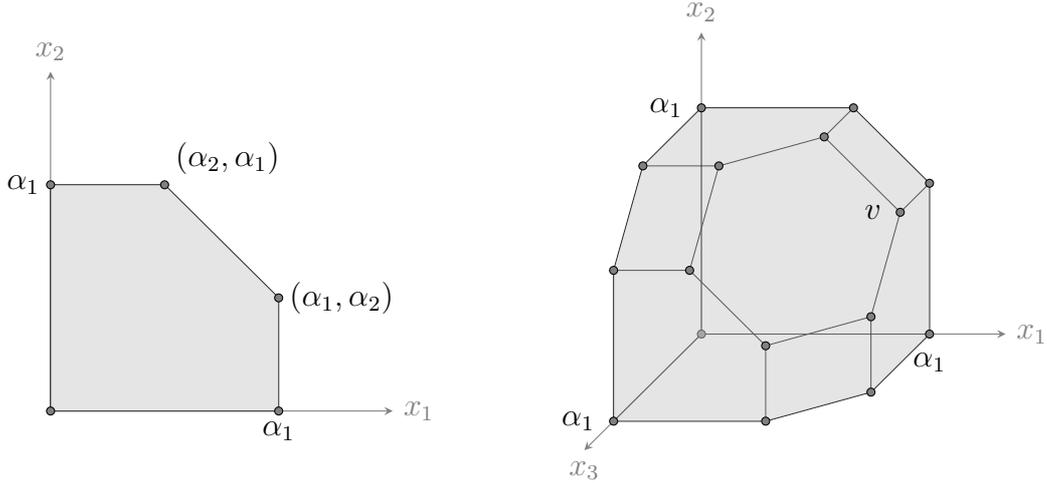
\begin{proposition}\label{prop:SIM-vertices}
  For $\alpha_1 \geq \cdots \geq \alpha_n \geq 0$ we have
  \[
    \Lambda(\alpha_1,\dots,\alpha_n) \ = \ \conv \left\{ \bigcup_{k=0}^n \cP_n (\alpha_1,\dots,\alpha_k,0,\dots,0) \right\} \enspace .
  \]
  Moreover, if $\alpha_1 > \cdots > \alpha_n > 0$, those points form the pairwise distinct vertices.
\end{proposition}
\begin{proof}
  A direct inspection shows that the points $(\alpha_1,\dots,\alpha_k,0,\dots,0)$ and their permutations are feasible.
  Assume $\alpha_1 > \cdots > \alpha_n > 0$.
  Then  $(\alpha_1,\dots,\alpha_k,0,\dots,0)$ is a vertex of $\Lambda(\alpha_1,\dots,\alpha_n)$ because it is the only feasible point that maximizes the linear objective function $(\alpha_1,\dots,\alpha_k,-1,\dots,-1)$.
  If the parameters $\alpha_i$ are not pairwise distinct, the claim follows from the generic case by continuity.
\end{proof}
\begin{example}
  For instance, for $\alpha_1 > \alpha_2 > 0$ the $2$-dimensional SIM-body
  \[
    \Lambda(\alpha_1,\alpha_2) \ = \ \conv \bigl\{ (0,0),\, (\alpha_1, 0),\, (0, \alpha_1),\, (\alpha_1,\alpha_2),\, (\alpha_2, \alpha_1) \bigr\}
  \]
  is a convex pentagon, which resembles a SIM-card used to authenticate subscribers on mobile phones; hence the name.
  In the telecommunication context the term \enquote{SIM} is short for \enquote{subscriber identity module}.
\end{example}

There are several immediate consequences of Proposition~\ref{prop:SIM-vertices}.
Recall that \emph{normally equivalent} polytopes, by definition, share the same normal fan.
In particular, such polytopes are affinely (and thus combinatorially) isomorphic.
\begin{corollary}\label{cor:SIM-combinatorics}
  Let $\alpha_1 \geq \cdots \geq \alpha_n \geq 0$.
  \begin{enumerate}[(a)]
  \item The SIM-body $\Lambda(\alpha_1,\dots,\alpha_n)$ is bounded; i.e., it is a polytope;
  \item that polytope is integral if and only if all parameters $\alpha_i$ are integers;
  \item we have $\dim\Lambda(\alpha_1,\dots,\alpha_n)=n$ if and only if $\alpha_1>0$;
  \item\label{it:SIM-combinatorics:equivalent} for pairwise distinct parameters $\alpha_i$ the SIM-body $\Lambda(\alpha_1,\dots,\alpha_n)$ is normally equivalent to $\Lambda(n,n-1,\dots,1)$;
  \item\label{it:SIM-combinatorics:orbits} the linear automorphism group of $\Lambda(n,n-1,\dots,1)$ is the symmetric group $\Sym(n)$, and it has precisely $n+1$ vertex orbits;
  \item\label{it:SIM-combinatorics:n-vertices} the number of vertices of $\Lambda(n,n-1,\dots,1)$ equals $\sum_{k=0}^n n!/k!$.
  \end{enumerate}
\end{corollary}
For $\alpha_1 > \cdots > \alpha_n > 0$ we call $\Lambda(\alpha_1,\dots,\alpha_n)$ a \emph{proper} SIM-body of \emph{(permutation) degree}~$n$.
By Corollary~\ref{cor:SIM-combinatorics}\eqref{it:SIM-combinatorics:equivalent}, there is only one combinatorial type of proper SIM-bodies for each degree.
That is, for combinatorial purposes it suffices to study the \emph{regular} SIM-body $\Lambda_n:=\Lambda(n,n-1,\dots,1)$ of rank $n$.
The sequence $a(n):=\sum_{k=0}^n n!/k!$ from Corollary~\ref{cor:SIM-combinatorics}\eqref{it:SIM-combinatorics:n-vertices}, which gives the number of vertices of $\Lambda_n$, occurs as A000522 in the On-Line Encyclopedia of Integer Sequences \cite{OEIS}.
The regular SIM-body $\Lambda_n$ occurs as \enquote{$Q_n$} in the work of Doker \cite[\S2.4]{Doker:phd}; general SIM-bodies are special cases of \emph{$Q$-polytopes} in~\cite{Doker:phd}.

\begin{example}\label{exmp:cube-simplex}
  For the most part we will focus on proper SIM-bodies, but there are some interesting special cases which are not proper.
  For instance, $\Lambda(1,1,\dots,1)=[0,1]^n$ is the unit cube, and $\Lambda(1,0,\dots,0)=\conv\{0,e_1,e_2,\dots,e_n\}$ is the standard simplex.
\end{example}

Next we will show that the $H$-description of $\Lambda(\alpha_1,\dots,\alpha_n)$ from Proposition~\ref{prop:projection} is nonredundant if the SIM-body is proper.
Moreover, we will determine the vertex-facet incidences.
Any nonempty set $I\subseteq[n]$ yields the inequality
\begin{equation}\label{eq:active:linear}
  \sum_{i \in I} x_i \ \leq \ \alpha_{1}+\dots + \alpha_k \enspace .
\end{equation}
Consider the vertex $v:=(\alpha_1,\dots,\alpha_\ell,0,\dots,0)$ of $\Lambda(\alpha_1,\dots,\alpha_n)$.
Then we say that the set $I$ with $\card{I}=k$ is \emph{active} at $v$ if the inequality \eqref{eq:active:linear} is tight at $v$, i.e.,
\[
  \sum_{i\in I} v_i  \ = \ \sum_{i\in I\cap[\ell]}\alpha_i \ = \  \alpha_{1}+\dots + \alpha_k \enspace ,
\]
or equivalently $I\cap[\ell]=[k]$.
The latter holds if and only if
\begin{equation}\label{eq:active:combinatorial}
  k\leq \ell \quad \text{and} \quad I=[k] \enspace .
\end{equation}

\begin{lemma}
  A proper SIM-body is a simple polytope, and it has exactly $2^n+n-1$ facets.
\end{lemma}
\begin{proof}
  Let $\alpha_1 > \cdots > \alpha_n > 0$.
  Then the SIM-body $\Lambda:=\Lambda(\alpha_1,\dots,\alpha_n)$ is $n$-dimensional.
  We need to show that each vertex is incident with exactly $n$ facets.
  As in our previous deliberations we consider the vertex $v:=(\alpha_1,\dots,\alpha_\ell,0,\dots,0)$.
  From \eqref{eq:active:combinatorial} we obtain exactly $\ell$ active sets at $v$.
  These contribute the same number of inequalities like \eqref{eq:active:linear} which are tight at $v$.
  Additionally, exactly the last $n-\ell$ nonnegativity constraints are tight at $v$.
  The facet count results from $2^n-1$ nonempty active sets and $n$ nonnegativity constraints.
\end{proof}

We are now able to prove our main result about the geometry of SIM-bodies.
\begin{theorem}\label{thm:generalized}
  Let $\alpha_1 \geq \cdots \geq \alpha_n \geq 0$.
  Then each edge of the homogeneous SIM-body $\Gamma(\alpha_1,\dots,\alpha_n)$ is parallel to the difference of two standard basis vectors.
\end{theorem}
A polytope with that property is called a \emph{generalized permutahedron} \cite{Postnikov:2009}.
\begin{proof}
  First, consider the special case where $\alpha_k=n-k+1$ for $k\in[n]$; i.e., $\Gamma_n:=\Gamma(n,n-1,\dots,1)$ is the regular homogeneous SIM-body of rank $n$ in $\RR^{n+1}$.
  We abbreviate $\bar \ell:=\ell{+}1 + \ell{+}2 + \dots + n$.
  Then $v:=(1,\dots,\ell,0,\dots,0,\bar \ell)$ is a vertex of $\Gamma_n$.
  There are $\ell$ active sets and $n-\ell$ (lifts of) nonnegativity constraints which define the $n$ facets incident with $v$.
  Let $w$ be a vertex adjacent to $v$.
  Then there are exactly $n-1$ facets incident to both vertices.
  Several cases arise.
  For instance, it may happen that $\ell-1$ of the facets from sets active at $v$ and all $n-\ell$ (lifts of) nonnegativity constraints are tight at $w$.
  Then $w=(1,\dots,\ell{-}1,0,\dots,0,\overline{\ell{-}1})$, and $v-w=\ell(e_\ell-e_{n+1})$.
  The same reasoning, with opposite signs, applies if the number of active sets increases when we walk from $v$ to $w$.
  It remains to discuss the situation when the number of facets from active sets incident with $v$ and $w$ is the same.
  This occurs if and only if $v$ and $w$ lie in the same orbit of the action of $\Sym(n)$.
  From Proposition~\ref{prop:SIM-vertices} we see that these vertices form the vertices of the permutahedron $\cP_n (n,\dots,\ell,0,\dots,0)$, after projection to $\RR^n$.
  So in this case the claim follows from Lemma~\ref{lem:permutahedron}.

  Now Corollary~\ref{cor:SIM-combinatorics}\eqref{it:SIM-combinatorics:equivalent} proves the claim for an arbitrary SIM-body which is proper, i.e., it satisfies $\alpha_1 > \cdots > \alpha_n > 0$.
  This leaves the degenerate cases where some, or all, of the parameters coincide.
  Then some edge lengths shrink to zero, but this does not affect the edge directions.
\end{proof}

\section{Optimal Deterministic Auctions}
\label{sec:auctions}
We now turn to our main topic.
A single buyer is interested in buying $n$ items.
The value of the buyer for item~$i \in [n]$ is denoted by $x_i$ and drawn from the uniform distribution on $[0,1]$. The value $x_i$ corresponds to the amount of money the buyer is willing to spend in order to receive item~$i$ and is the private information of the buyer.
A \emph{direct revelation mechanism} is a tuple $M = (a,r)$ of functions $a : [0,1]^n \to [0,1]^n$ and $r : [0,1]^n \to \RR$.
The mechanism elicits a bid vector $x' \in [0,1]^n$ from the buyer.
It then sells to the buyer at a price of $r(x')$ the following lottery: For each item $i \in [n]$ the buyer receives item~$i$ with probability $a_i(x')$ where the probabilities for different items are independent of each other.
We call a direct revelation mechanism an \emph{auction}.
An auction is \emph{deterministic} if $a(x') \in \{0,1\}^n$ for all $x' \in [0,1]^n$. The price $r(x')$ payed by the buyer is also called the \emph{revenue} of the auction. The buyer is risk-neutral and has additive valuations, meaning that their utility is equal to the expected value of the items received by the mechanism minus the payment. Fixing a direct revelation mechanism $M = (a,r)$, the utility of the buyer with valuation $x$ who bids $x'$ is given by
\[
u(x' \mid x) \ = \ -r(x') +\sum_{i=1}^n a_i(x')x_i \enspace.
\]
We impose the following two assumptions on the design space of direct revelation mechanisms.
First, we require that the buyer prefers to bid according to their actual valuation vector $x$ over claiming any other bid vector $x' \neq x$. Bidding the true valuation is also called \emph{truthful} bidding. A mechanism is called \emph{incentive compatible} if bidding truthfully is an optimal strategy of the buyer, i.e.,
\begin{equation}
\label{eq:IC}
u(x \mid x) \geq u(x'\mid x) \quad \text{for all $x,x' \in [0,1]^n$} \enspace.\tag{IC}	
\end{equation}
Second, we require that the mechanism is \emph{individually rational}. This property (which is sometimes also called \emph{voluntary participation}) requires that for every realization~$x$ of the buyers' valuations, they are guaranteed to leave the auction with a nonnegative utility when bidding truthfully. Formally, the individual rationality constraint reads
\begin{align}
\label{eq:IR}
u(x\mid x) \geq 0 \quad \text{for all $x \in [0,1]^n$} \enspace. \tag{IR}	
\end{align}
Let $u(x) = u(x \mid x)$ for all $x \in [0,1]^n$.
Our goal is to design an auction that maximizes the expected revenue under the constraints that the auction is incentive compatible and individual rational, i.e., we are interested in solving
\begin{alignat*}{2}
\sup  \quad &\EE[r(x)] \\
\text{s.t.} \quad & a : [0,1]^n \to [0,1]^n \text{ and } r : [0,1]^n \to \RR\enspace,\\
& M= (a,r) \text{ satisfies \eqref{eq:IC} and \eqref{eq:IR}} \enspace,
\end{alignat*}
where the expectation in the objective is taken over all draws of $x_i$ from the uniform distribution in $[0,1]$. 
The following result of Rochet~\cite{Rochet1985} characterizes the set of mechanisms satisfying \eqref{eq:IC} in terms of convexity.
\begin{proposition}\label{prop:rochet}
  The auction $M = (a,r)$ satisfies \eqref{eq:IC} if and only if the utility $u(x)$ is convex and differentiable almost everywhere with
  \begin{align*}
    {\partial u(x)}/{\partial x_i} = a_i(x) \quad \text{ for all $i \in [n]$ and almost all $x \in [0,1]^n$}\enspace. 
  \end{align*}
\end{proposition}

Proposition~\ref{prop:rochet} allows to parameterize the space of all auctions satisfying \eqref{eq:IC} by the utility functions $u : [0,1]^n \to \RR$ that they impose on the buyer.
The constraint that the auction is \eqref{eq:IR} can then be incorporated by requiring that $u : [0,1]^n \to \RR_{\geq 0}$.
Together with the observation that $r(x) = \sum_{i=1}^n a_i(x)x_i- u(x)$, we obtain the following optimization problem to determine a revenue maximizing auction satisfying \eqref{eq:IC} and \eqref{eq:IR} 
\begin{equation}\label{eq:primal} 
\begin{aligned}
\sup \quad &\EE\Biggl[\sum_{i=1}^n {\partial u(x)}/{\partial x_i} \ x_i \,-\, u(x)\Biggr] \\
\text{s.t.} \quad & u : [0,1]^n \to \RR_{\geq 0} \text{ is convex}\enspace,\\
& 0 \leq \frac{\partial u(x)}{\partial x_i} \leq 1 \quad \text{ for all $i \in [n]$ and almost all $x \in [0,1]^n$} \enspace. 
\end{aligned}
\end{equation}

For the following, let us assume that the primal solution is a deterministic auction. 
Formally, let $\{p_I\}_{I \subseteq [n]}$ be a \emph{price schedule} where we require $p_{\emptyset} = 0$.
A price schedule defines a unique deterministic auction where the buyer reports their valuation and is assigned a subset $I \subseteq [n]$ of items that maximizes their utility which is the valuation attached to the items contained in the set minus the price $p_I$ for set $I$.
Consequently, the utility of the buyer for this auction is given by the function
\begin{equation}
  \label{eq:tropical_polynomial}
  u(x) \ = \ \max \biggSetOf{ \sum_{i \in I} x_i - p_I  }{I \subseteq [n]} \enspace ,
\end{equation}
which is the evaluation of a $\max$-tropical polynomial of degree $n$.
The set of points in~$\RR^n$ where that maximum is attained at least twice is the \emph{tropical hypersurface} $\cT(u)$; see \cite[\S3.1]{Maclagan+Sturmfels:2015} or \cite[\S1.1]{ETC}.
This agrees with the set of points where the evaluation $x\mapsto u(x)$ is nondifferentiable.
For any price schedule $\{p_I\}_{I \subseteq [n]}$, the tropical hypersurface of $u$ subdivides the $n$-dimensional unit cube $[0,1]^n$ into polytopes $\{D_I\}_{I \subseteq [n]}$ such that $D_I$ is the set of valuations $x$ for which the maximum in \eqref{eq:tropical_polynomial} is attained at $I$.
Formally, we have
\[
  D_I \ = \ \Biggl\{x \in [0,1]^n \;\Bigg\vert\; \sum_{i \in I} x_i - p_I \geq \sum_{j \in J} x_j - p_J \text{ for all } J \subseteq [n] \setminus I \Biggr\} \enspace.	
\]
The polytopes $D_I$ are called the \emph{regions} of $\{p_I\}_{I \subseteq [n]}$.
In fact, they are the intersections of the unit cube with the regions of linearity of the tropical hypersurface $\cT(u)$.

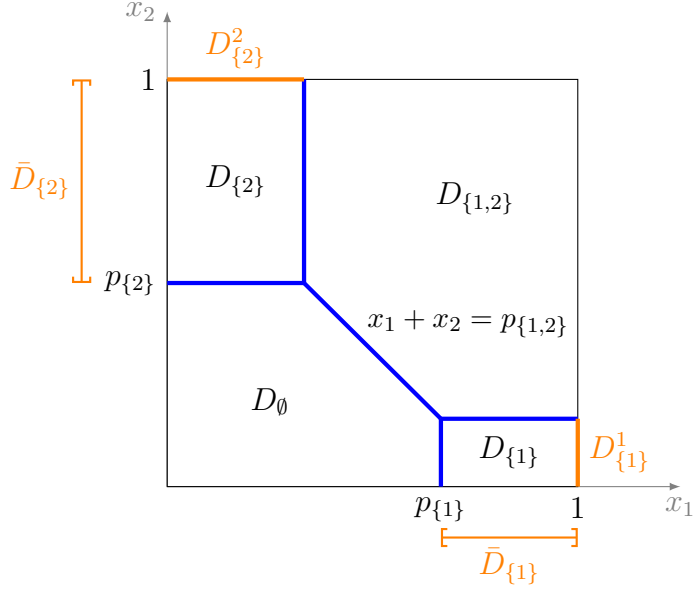
\begin{figure}[th]
  \tikzset{
    axis/.style={thin,gray,-latex},
    tropical/.style={ultra thick,blue}, 
    border/.style={{[-]}, orange, thick}
  }
  \begin{tikzpicture}[scale = 0.9]
    \draw[axis] (0,0)--(7.5,0) node[below]{$x_1$};
    \draw[axis] (0,0)--(0,7) node[left]{$x_2$};
    \draw (0,0)--(6,0)--(6,6)--(0,6)--cycle;

    \draw[tropical] (4,0)--(4,1);
    \draw node[below] at (4,0) {$p_{\{1\}}$};
    \draw node[below] at (6,0) {$1$};
    \draw[tropical] (0,3)--(2,3);
    \draw node[left] at (0,3) {$p_{\{2\}}$};
    \draw node[left] at (0,6) {$1$};
    \draw[tropical] (2,3)--(4,1);
    \draw[tropical] (2,3)--(2,6);
    \draw[tropical] (4,1)--(6,1);
    
    \draw[border] (-1.25,3)--(-1.25,6);
    \draw node[left, orange] at (-1.25,4.5) {$\bar D_{\{2\}}$};
    \draw[ultra thick, orange] (0,6)--(2,6);
    \draw node[above, orange] at (1,6) {$D_{\{2\}}^2$};
    
    \draw[border] (4,-.75)--(6,-.75);
    \draw node[below, orange] at (5,-.75) {$\bar D_{\{1\}}$};
    \draw[ultra thick, orange] (6,0)--(6,1);
    \draw node[right, orange] at (6,.5) {$D_{\{1\}}^1$};
    
    \draw node at (4.4,2.4) {\small $x_1 + x_2 = p_{\{1,2\}}$};
    \draw node at (1.5,1.25) {$D_{\emptyset}$};
    \draw node at (5,0.5) {$D_{\{1\}}$};
    \draw node at (1,4.5) {$D_{\{2\}}$};
    \draw node at (4.5,4.25) {$D_{\{1,2\}}$};
  \end{tikzpicture}
  \caption{Tropical hypersurface and regions of the mechanism with price schedule $p_{\{1\}} = \frac{2}{3}$, $p_{\{2\}} = \frac{1}{2}$ and $p_{\{1,2\}} = \frac{5}{6}$, which is submodular.
    If the valuation of the buyer lies in $D_I$, their utility is maximized by buying the items in $I$.
    The polytope $\bar D_I$ is the projection of $D_I$ onto the coordinates in $I$, and $D_I^i$ is the intersection of $D_I$ with the facet $x_i=1$ of the cube.}
  \label{fig:tropical}
\end{figure}

Manelli and Vincent~\cite{ManelliVincent:2006} describe necessary conditions for the optimality of a price schedule for an arbitrary distribution of the valuations.
In the uniform case their criterion \cite[Theorem 1]{ManelliVincent:2006} says that an optimal price schedule $\{p_I\}_{I\subseteq[n]}$ necessarily satisfies
\begin{equation} \label{eq:deficiency}
  \vol_n{D_I} \, - \, \frac{1}{n+1} \sum_{i \in I} \vol_{n-1}{D_I^i} \ = \ 0
\end{equation}
for all $I \neq \emptyset$ with $\vol_n D_I > 0$, where $x_{-i}$ skips the $i$-th coordinate of $x$, and $D_I^i = \{x_{-i} \mid (1, x_{-i}) \in D_I \}$ is the \emph{$i$-th upper boundary} of the region $D_I$.
The following lemma says that $D_I^i$, which is defined via intersecting the region $D_I$ with one facet of the unit cube, agrees with a full projection of $D_I$ by leaving out the $i$-th coordinate.
\begin{lemma}\label{lem:upwards_closed}
  Let $\{p_I\}_{I\subseteq[n]}$ be an arbitrary price schedule.
  Further, let $x \in [0,1]^n$ and $I \subseteq[n]$ such that $x \in D_I$.
  Then $x_{-i}\in D_I^i$ for all $i \in I$.
\end{lemma}
\begin{proof}
  Fixing $i\in I$ our goal is to show that $x':=(1, x_{-i}) \in D_I$.
  As $x \in D_I$, we have
  \[
    x_i + \sum_{j \in I \setminus \{i\}} x_j - p_I \ \geq \ \sum_{j \in J} x_j - p_J
  \]
  for all $J \subseteq[n]$.
  If $i \notin J$ the left hand side of that inequality increases if we exchange $x$ for $x'$, whereas the right hand side stays the same.
  Otherwise, if $i \in J$, we have:
  \[
    \sum_{j \in I} x'_j- p_I \ = \ 1+(x_i - x_i) + \sum_{j \in I \setminus \{i\}} x_j - p_I \ \geq \ 1-x_i + \sum_{j \in J} x_j - p_J \ = \ \sum_{j \in J} x'_j - p_J \enspace .
  \]
  In both cases we obtain $x' \in D_I$.
\end{proof}
The deficiency condition \eqref{eq:deficiency} leads us to analyzing regions of positive volume.
\begin{lemma}\label{lem:relevantPrice}
  Let $\{p_I\}_{I\subseteq[n]}$ be an arbitrary price schedule.
  Further, let $I \subseteq [n]$ with $\vol_n D_I > 0$.
  Then $p_I < p_J$ for all $J \supsetneq I$.
\end{lemma}
\begin{proof}
  By definition $\vol_n D_I > 0$ if and only if there exists $x \in [0,1]^n$ such that
  \[
    \sum_{i \in I} x_i - p_I \ > \ \sum_{j \in J} x_j - p_J
  \]
  for all $I \neq J \subseteq[n]$.
  For $J \supsetneq I$ we have $\sum_{j \in J} x_j \geq \sum_{i \in I} x_i$, whence we conclude $p_I < p_J$.
\end{proof}
Let $\pi_I$ be the projection onto the coordinates in $I\subseteq[n]$. 
The next lemma shows a recursive property of deterministic auctions.
Under the condition that for a given bundle $I$ all regions of its sub-bundles $J \subseteq I$ have positive volume, the projections of those regions only depend on the prices of sub-bundles of $I$.
This will be useful later, when we want to compute certain price schedules.
\begin{lemma}\label{lem:shape}
  Let $\{p_I\}_{I \subseteq[n]}$ be an arbitrary price schedule.
  Further, let $I \subseteq [n]$ with $\vol_n D_J > 0$ for all $J \subseteq I$.
  Then 
  \[
    \pi_I(D_J) \ = \ \biggSetOf{x \in [0,1]^I}{\sum_{j \in J} x_j - p_{J} \geq \sum_{j \in K} x_j - p_{K} \text{ for all } K \subseteq I} \enspace .
  \]
\end{lemma}
\begin{proof}
  Let $J$ be a subset of $I$.
  As in Lemma \ref{lem:upwards_closed} we can check that for $x \in D_J$ also $(0,x_{-j})$ is contained in $D_J$ if $j \notin J$.
  This implies that $\pi_I(D_J)$ arises from $D_J \cap \SetOf{x}{x_i = 0 \text{ for } i \notin I}$ by skipping coordinates.
  So let $x \in [0,1]^n$ with $x_i = 0$ for $i\notin I$.
  From Lemma \ref{lem:relevantPrice} we get $p_{K\cap I} \leq p_K$ for each $K \subseteq [n]$ and hence
  \[
    \sum_{j \in K} x_j - p_K \ \leq \ \sum_{j \in K\cap I} x_j - p_{K\cap I} \enspace ,
  \]
  which finishes the proof.
\end{proof}

We are particularly interested in the following special class of price schedules.
A price schedule $p$ is \emph{submodular}, if
\begin{equation}\label{eq:submodular}
  p_{I\cup J} + p_{I\cap J} \ \leq \ p_I + p_J \enspace .
\end{equation}
for all $I,J \subseteq[n]$; see Figure~\ref{fig:tropical} for an example with $n=2$ items.
The submodularity of a price schedule is a desirable property for the seller.
Namely, it prevents a buyer interested in the items in $I \cup J$ from first buying $I$ and $J$ separately at prices $p_I$ and $p_J$ and then returning the (duplicate) items in the intersection $I \cap J$ at price $p_{I \cap J}$, instead of buying the set $I \cup J$ at price $p_{I \cup J}$ directly.
For a more comprehensive discussion see~\cite[Definition~4]{ManelliVincent:2006}.
The regions of submodular price schedules can be written as products of some of their projections.
\begin{lemma}\label{lem:product}
  Let $\{p_I\}_{I\subseteq[n]}$ be a submodular price schedule.
  Then  $D_I = \pi_I(D_I) \times \pi_{I^c}(D_I)$ for all $I \subseteq[n]$ and $I^c = [n] \setminus I$.
  Moreover, $D_I^i = \pi_{I\setminus \{i\}} ( \pi_I(D_I)) \times \pi_{I^c}(D_I)$ for all $i \in I$.
\end{lemma}
\begin{proof}
  The first claim has been proved by Vincent and Manelli~\cite[Lemma~2]{ManelliVincent:2006}.
  The second claim then follows from Lemma \ref{lem:upwards_closed}.
\end{proof}
To simplify our notation, we abbreviate $\bar D_I = \pi_I(D_I)$ and $\bar D_I^i = \pi_{I\setminus \{i\}}(\pi_I(D_I))$.
Let us further define
\[
  \delta_n(I) \ \coloneqq \ \vol_{\card{I}}{\bar D_I} \, - \, \frac{1}{n+1} \sum_{i\in I} \vol_{\card{I} - 1}{\bar D_I^i} \enspace.
\]
It is a consequence of Lemma \ref{lem:upwards_closed} that $\delta_n([n])$, which agrees with the left hand side of \eqref{eq:deficiency}, is the \emph{deficiency} of $D_{[n]}$ with respect to $\frac{1}{n+1}$ in the sense of \cite[Definition 4.5]{GK+duality:2018}.
The expected revenue for given prices $p = \{p_I\}_{I\subseteq[n]}$ is
\begin{equation}\label{eq:revenue}
  \Rev(p) \ = \ \sum_{I \subseteq [n]} p_I \cdot \vol_n D_I \enspace ,
\end{equation}
and this depends smoothly on $p$.
\begin{lemma}\label{lem:RevenueDerivative}
  Let $\{p_I\}_{I\subseteq[n]}$ be a submodular price schedule.
  Then the partial derivative of the expected revenue with respect to $p_I$ is given as
  \[
    {\partial\Rev(p)}/{\partial p_I} \ = \ \delta_n(I) \cdot \vol_{\card{I^c}} \pi_{I^c} (D_I)  \enspace .
  \]
\end{lemma}
\begin{proof}
  The statement follows from \cite[Theorem 1]{ManelliVincent:2006} and Lemma~\ref{lem:product}; see Equation~\eqref{eq:deficiency}.
\end{proof}
Since we are interested in maximizing the revenue it is natural to analyze the critical points of the revenue function.
However, we also need to take care of boundary effects.
To this end we investigate bundles where the submodularity condition \eqref{eq:submodular} is attained with equality.
If $p_{I \cup J} + p_{I \cap J} = p_I + p_J$ we call $I\cup J$ and $I\cap J$ \emph{left tight}, while we call $I$ and $J$ \emph{right tight}.
A bundle is \emph{tight} if it is left or right tight.
The following describes the necessary critical point conditions and may be seen as a refinement of \cite[Theorem 1]{ManelliVincent:2006} for uniform valuations and submodular price schedules.
\begin{proposition}\label{prop:necessary}
  Let $\{p_I\}_{I\subseteq[n]}$ be a submodular price schedule which maximizes the expected revenue among all submodular price schedules.
  Then we have, for all $I\subseteq[n]$:
  \begin{enumerate}
  \item\label{cond:notTightPosVol} If $I$ is not tight and $\vol_n D_I > 0$, then $\delta_n(I) = 0$;
  \item\label{cond:notTightTrivVol} if $I$ is not tight and $\vol_n D_I = 0$, then $\delta_n(I) \geq 0$;
  \item if $I$ is just left tight, then $\delta_n(I) \geq 0$;
  \item if $I$ is just right tight and $\vol_n D_I > 0$, then $\delta_n(I) \leq 0$.
  \end{enumerate}
\end{proposition}
If $I$ is right tight and $\vol_n D_I = 0$, there is no condition on $\delta_n(I)$. The same applies if $I$ is left tight as well as right tight.
\begin{proof}
  First assume that  $I \subseteq [n]$ is not tight.
  Then the prices involving $I$ are strictly submodular, i.e., \eqref{eq:submodular} is strictly satisfied for all $J \neq I$.
  Hence we can slightly increase or decrease $p_I$, without loosing the submodularity.
  Suppose that $\vol_n D_I > 0$.
  As $p$ is optimal and using Lemma \ref{lem:RevenueDerivative}, this forces $\delta_n(I) = 0$, which says that $p$ is a critical point of the revenue function.
  
  Now, if $I$ is not tight but $\vol_n D_I = 0$ the volume of $D_I$ (and thus the revenue) can only change if we decrease $p_I$.
  The optimality of the prices entails $\delta_n(I) \geq 0$, again using Lemma \ref{lem:RevenueDerivative}.
	
  We are left with the analysis of those bundles which are tight.
  So we consider $I \neq J$ with $p_{I\cup J} + p_{I \cap J} = p_I + p_J$.
  Decreasing $p_{I\cup J}$ or $p_{I \cap J}$ keeps the submodularity.
  Therefore, as in the case $\vol_n D_I = 0$ above, we conclude $\delta_n({I\cup J}) \geq 0$ and $\delta_n({I \cap J}) \geq 0$.
  This takes care of the left tight sets.
  For $I$ right tight with $\vol_n D_I > 0$ the situation is symmetric, but with opposite signs.
\end{proof}

Giannakopoulos and Koutsoupias \cite{GK+duality:2018} used a generalization of the deficiency $\delta_n(I)$, which is a property of the region $D_I$, to arbitrary subsets of $[0,1]^n$.
In Proposition \ref{prop:necessary} we used the condition $\vol_n D_I = 0$.
Note however, that if $\vol_n D_I = 0$ we can set $p_I$ arbitrarily high without affecting the auction.
Following \cite[Section 6.2]{GK+duality:2018} a price schedule is \emph{normalized} if $p_I = \min \SetOf{p_J}{J \supsetneq I}$, whenever $\vol_n D_I = 0$.
This entails unique prices, and so price schedules bijectively correspond to auctions.
Throughout we assume that our price schedules are normalized, unless stated otherwise.

A set of bundles $\cI \subseteq 2^{[n]}$ is a \emph{selling candidate} for $n$ items if, firstly, $[n]\in \cI$, secondly, $\cI$ contains all singletons, and, thirdly, for all $I\in\cI\setminus\{[n]\}$ and $J \subseteq I$ we have $J \in \cI$.
That is to say, $\cI\setminus\{[n]\}$ is an abstract simplicial complex with vertex set $[n]$.
In fact, that simplicial complex can be seen geometrically embedded into the vertex figure of the vertex $\1$ of the unit cube $[0,1]^n$, which is also a vertex of the region $D_{[n]}$.
We say that a price schedule $\{p_I\}_{I\subseteq[n]}$ \emph{admits} a selling candidate $\cI$ if and only if $\vol_n D_I>0$ for all $I \in \cI$.
Moreover, the price schedule is \emph{critical} with respect to $\cI$ if additionally $\delta_n(I) = 0$ for all $I\in\cI \setminus \{\emptyset \}$.
The following is our main technical result; it complements \cite[Lemma 6.5]{GK+duality:2018} and generalizes \cite[Appendix A]{GK+duality:2018} to nonsymmetric prices.
Our proof is constructive; the special case of most interest to us is spelled out as Algorithm \ref{algo:SJA-prices} in Section~\ref{sec:computing}.

\begin{theorem}\label{thm:uniqueCritPrice}
  For each selling candidate $\cI$ there is at most one critical price schedule which is submodular.
  Moreover, it satisfies $p_I = p_J$ whenever $\card{I} = \card{J}$ and $I,J \in \cI$.
\end{theorem}
\begin{proof}
  Let $I\in\cI$, where $\cI$ is a selling candidate for $n$ items.
  We seek to determine a price for the bundle $I$, and we will do so by induction on the cardinality of $I$.
  Let $i\in I$, whence the singleton $\{i\}$ is contained in $\cI$.
  We set $p_{\{i\}} = \tfrac{n}{n+1}$.
  Then, by Lemma \ref{lem:shape},
  \begin{equation}
    \delta_n({\{i\}}) \ = \ \vol_1(\bar D_{\{i\}})-\tfrac{1}{n+1}\vol_0(\bar D_{\{i\}}^i) \ = \ (1-p_{\{i\}})-\tfrac{1}{n+1} \ = \ 0 \quad \text{for all } i \in [n] \enspace .
  \end{equation}
  This computation also shows that $p_{\{i\}}$ is uniquely determined.
  This case serves as the basis of the induction. 
  
  Now let $I \in \cI$ with $\card{I}\geq 2$, and assume that $p_J$ is fixed for all $J \subsetneq I$ such that the following conditions are satisfied: $\delta_n(J) = 0$, $p_J = p_{J'}$ whenever $\card{J} = \card{J'}$, and $\{p_J\}_{J \subsetneq I}$ is submodular.
  We will now construct $p_I$ such that $\delta_n(I) = 0$ and $\{p_J\}_{J \subseteq I}$ is submodular, or decide that such a price $p_I$ doesn't exist.
  
  Since $p_{\{i\}} = \frac{n}{n+1}$ for all $i \in [n]$, we have $(1, x_{-i}) \notin D_{\emptyset}$ for all $x_{-i} \in [0,1]^{\card{I} -1}$.
  Further, by Lemma \ref{lem:shape} we know that $\{\pi_I(D_J)\}_{J \subseteq I}$ is a polyhedral complex supported on $[0,1]^I$.
  Therefore we get
  \begin{equation}\label{eq:borderSum}
    \sum_{\{i\} \subseteq J \subseteq I} \vol_{\card{I}-1} \pi_I (D_{J}^i) \ = \ 1 \enspace ,
  \end{equation}
  for all $i \in I$.
  The last equation uses the fact that for $i \in I$ and $J \subset I$ with $i \notin J$ we have $D_J^i = \emptyset$.
  This can be seen as follows.
  Let $x\in[0,1]^n$ satisfying $x_i=1$. 
  Then $x \in D_J^i$ if and only if $J = \argmax \SetOf{\sum_{i \in I} x_i - p_I}{I \subseteq [n]}$.
  But from the submodular inequality $p_\emptyset + p_{J\cup\{i\}} \leq p_{\{i\}}+p_J$ and $p_{\{i\}} < 1$, we have $\sum_{j\in J} x_j-p_J < \sum_{j\in J} x_j-(p_{J\cup\{i\}}-1) = \sum_{j\in J\cup\{i\}} x_j-p_{J\cup\{i\}}$.
  Hence, $D_J^i = \emptyset$.
  Now, using \eqref{eq:borderSum} and the fact that $\delta_n(J) = 0$ for all $J \subsetneq I$, the condition $\delta_n(I) = 0$ is equivalent to
  \begin{equation}\label{eq:sumOfVols}
    \sum_{\emptyset \neq J \subseteq I} \vol_{\card{I}} \pi_I(D_J) \ = \ \sum_{\emptyset \neq J \subseteq I} \frac{1}{n+1} \cdot \sum_{j\in J} \vol_{\card{I}-1} \pi_I (D_{J}^j) \ = \ \frac{\card{I}}{n+1} \enspace .
  \end{equation}
  From Lemma \ref{lem:shape} we get $\vol_{\card{I}} \pi_I(D_\emptyset) = 1 - \sum_{\emptyset \neq J \subseteq I} \vol_{\card{I}} \pi_I(D_J)$.
  Hence \eqref{eq:sumOfVols} becomes
  \begin{equation}\label{eq:volDEmpty}
    \vol_{\card{I}} \pi_I(D_\emptyset) \ = \ 1-\frac{\card{I}}{n+1} \enspace .
  \end{equation}
  Note that $\pi_I(D_{\emptyset}) = \smallSetOf{x \in [0,1]^I}{\sum_{j \in J} x_j \leq p_J \text{ for } J \subseteq I}$ due to Lemma \ref{lem:shape}.
  Since the prices $p_J$ are fixed for $J \subsetneq I$, the volume $\vol_{\card{I}} \pi_I(D_\emptyset)$ depends differentiably on the variable $p_I$.
  We define $\hat p_I$ to be the maximal value for $p_I$, such that all submodularity conditions of the form
  \[
    p_I + p_{J \cap J'} \ \leq \ p_J + p_{J'}
  \]
  are satisfied for all $J, J' \subseteq I$ with $J \cup J' = I$.
  Choosing $J, J'$ such that $J \cap J' = \emptyset$, we get
  \[
    \hat p_I \ \leq \ \max \biggSetOf{ \sum_{i\in I} x_i }{ x \in [0,1]^I  \text{ and } \sum_{j \in J} x_j \leq p_J \text{ for all } J \subsetneq I } \enspace .
  \]
  Therefore, the volume $\vol_{\card{I}} \pi_I(D_\emptyset)$ is strictly increasing on the interval $[0 ,\hat p_I]$.
  Hence, there is at most one solution to \eqref{eq:volDEmpty} in this interval.
  If there is no solution in this interval, there can not be a price $p_I$ satisfying both $\delta_n(I) = 0$ and the submodularity of $\{p_J\}_{J \subseteq I}$.
  On the other hand, if the solution exists, it satisfies both properties by the way we defined it.
  Further, since the computation of the solution does not depend on prices $p_J$ for bundles $J$ of higher or equal cardinality than $I$, the solution is the same for all $I'$ with $\card{I'} = \card{I}$.
  Note however, that the solution for $p_I$, may be strictly lower than a price $p_J$ for some $J \subseteq I$.
  In that case $D_J$ is empty, whence $\vol_{\#J}(D_J)=0$, and there is no submodular critical price schedule for the selling candidate $\cI$.
    
  For the case $I = [n]$ we assume that $p_J$ is fixed for all $J \in \cI \setminus \{[n]\}$ and we also fix $p_J = \card{J}$ for all $J \notin \cI$.
  This way, $D_J = \emptyset$ for $J \notin \cI$ and hence $\delta_n(J)=0$ is trivially satisfied.
  We have that $\{D_J\}_{J\subseteq [n]}$ is  a polyhedral complex supported on $[0,1]^n$.
  Moreover, for $i \in [n]$ and $J \subset [n]$  with $i\notin J$, we get $D_J^i = \emptyset$, since for $J \in \cI$ the prices $p_J$ are submodular, and for $J \notin \cI$ the statement is trivially satisfied.
  Therefore \eqref{eq:borderSum} still holds and we can apply the rest of the proof to $D_{[n]}$.
  In order to normalize the price schedule, we will redefine $p_J = p_n$ for all $J \notin \cI$.
  This does not change the region $D_{\emptyset}$ whence the computation of $p_n$ stays consistent.
\end{proof}

\begin{remark}\label{rem:solvable}
  We call the candidate set $\cI$ \emph{solvable} if there is a solution of \eqref{eq:volDEmpty} in the interval $[0,\hat p_I]$ for each bundle $I\in\cI$.
  However, even if $\cI$ is solvable, a submodular critical price schedule may fail to exist.
  This happens if there exists $I \in \cI$ whose solution in the interval $[0,\hat p_I]$ is smaller than the price $p_J$ for some $J \subsetneq I$ with $J\in\cI$.
  This means that $\vol_n(D_J)$ vanishes despite that the bundle $J$ is intended to be sold.
  This can always be fixed by removing bundles from the selling candidate.
\end{remark}





Let $\cI$ be a selling candidate for $n$ items for which a submodular critical price schedule $\{p_I\}_{I\subseteq[n]}$ exists; the latter is unique by Theorem~\ref{thm:uniqueCritPrice}.
Recall that a price schedule induces a utility function $u(x)$ as in \eqref{eq:tropical_polynomial}, which itself defines an auction.
To be precise, the functions $a : [0,1]^n \to \{0,1\}^n$ and $r : [0,1]^n \to \RR$ of the auction are given by
\begin{equation}
  a(x) = I \quad \text{and} \quad r(x) = p_I \qquad \text{for } x\in D_I \enspace;
\end{equation}
where we identify $I \subseteq [n]$ with its characteristic vector in $\{0,1\}^n$.
The functions $a$ and $r$ are well-defined as the regions $D_I$, which depend on the prices $p_I$, form a polyhedral complex which covers the unit cube.
We call the auction $M(\cI):=(a,r)$ the \emph{generalized straight jacket auction} (GSJA) for $n$ items with selling candidate $\cI$.

Next, we are focusing on symmetric price schedules.
That is, the prices only depend on the cardinality of the sets $I \in \cI$.
We define the \emph{straight jacket auction} (SJA) as follows.
Let $k\in[n]$ be the maximal number for which submodular, symmetric critical prices exist for the selling candidate $\cI_{k,n}:=\tbinom{[n]}{1}\cup\dots\cup\tbinom{[n]}{k}\cup\{\emptyset,[n]\}$.
Then we call $M(\cI_{k,n})$ the \emph{straight jacket auction} (SJA) for $n$ items.
This deviates slightly from the definition of Giannakopoulos and Koutsoupias \cite{GK+duality:2018}.
They proved that SJA is optimal among all auctions, deterministic or nondeterministic, for the uniform distribution and $n\leq 6$ items.
Yet from their definition \cite[Definition 4.1]{GK+duality:2018} it is not clear if such an auction exists for all $n$.
Our definition avoids this technical difficulty since $\cI_{1,n}$ always admits a (symmetric) submodular critical price schedule.
We say that the difference $n-k-1$ is the \emph{gap} of the SJA with $n$ items.
If the maximal selling candidate $\cI_{n-1,n}$ is solvable, then the gap counts the layers $\tbinom{[n]}{n-1},\dots,\tbinom{[n]}{k+1}$ which must be removed to obtain a submodular critical price schedule.

For the remainder of this section, we slightly abuse notation writing $p_k$ for the price for all sets of cardinality $k$.
Moreover, we will refer to the price vector $(p_1,\dots,p_n)$ associated to SJA as the \emph{SJA-prices}.
The next proposition shows that SJA is a candidate for the optimal auction.
\begin{proposition}\label{prop:SJA-solvable}
  Let $n \in \NN$ such that $\cI_{n-1,n}$ is solvable.
  Then SJA for $n$ items satisfies all necessary critical point conditions (1)--(4) in Proposition \ref{prop:necessary}.
\end{proposition}
\begin{proof}
  Condition \eqref{cond:notTightPosVol} is satisfied, by the way we construct SJA, see the proof of Theorem \ref{thm:uniqueCritPrice}.
  We want to check condition \eqref{cond:notTightTrivVol}.
  Let $p = (p_1, \dots, p_n)$ be the vector of SJA-prices, and $\cI_{k,n}$ be the corresponding candidate set.
  We need to show that $\delta_n([\ell]) \geq 0$ for all $\ell \in \{k+1,\dots,n-1\}$.
  Let us fix such a $\ell$.
  Now $\cI_{\ell,n}$ is solvable since $\cI_{n-1,n}$ is.
  Therefore, there is a price $q_{\ell} \in [0,\hat p_{\ell}]$ satisfying $\vol_{\ell} \pi_{[\ell]}(D_{\emptyset}) = 1 - \tfrac{\ell}{n+1}$.
  Yet $\cI_{\ell,n}$ fails to admit a critical price schedule, and so $p_n < q_{\ell}$.
  Since we assume that $p$ is normalized, we have $p_{\ell} = p_n$, whence $\vol_{\ell} \pi_{[\ell]}(D_{\emptyset}) < 1 - \tfrac{\ell}{n+1}$ and $\delta_n([\ell]) > 0$.
  The conditions (3) and (4) are implied because (1) and (2) hold for all bundles.
\end{proof}
\begin{remark}
  If $\cI_{n-1,n}$ is solvable, our definition of SJA coincides with \cite[Definition 4.1]{GK+duality:2018}.
  More precisely, the solvability of $\cI_{n-1,n}$ is equivalent to SJA of \cite[Definition 4.1]{GK+duality:2018} being well-defined and submodular.
  Note that submodularity is claimed in \cite[Lemma 6.2]{GK+duality:2018}, but the proof given has a flaw in the final line: the inequality \enquote{$(r-1)p_{r-2} \leq (r-2)p_{r-1}$} needs to be reversed.
\end{remark}

The next two lemmas are useful for the computation of the SJA-prices and their expected revenue.
They both establish connections between the SIM-bodies studied in Section~\ref{sec:sim} and the regions of submodular and symmetric price schedules.
\begin{lemma}\label{lem:DEmptyisSIM}
  Let $\{p_k\}_{k\in[n]}$ be a submodular and symmetric price schedule and $I \subseteq [n]$.
  Then the projected region $\pi_I(D_{\emptyset})$ equals the SIM-body $\Lambda(p_1, p_2 - p_1, \dots, p_{\card{I}} - p_{\card{I} - 1})$, provided $p_1\leq 1$.
\end{lemma}

\begin{proof}
  We have
  \begin{align*}
    \pi_I(D_{\emptyset}) \ &= \ \left\{x \in [0,1]^I \;\Bigg\vert\;  \sum_{j \in J} x_j \leq p_k \text{ for all $J \in \tbinom{I}{k}$, $k \geq 1$} \right\}\\
                           &= \ \Lambda(p_1,p_2 - p_1,\dots,p_{\card{I}} - p_{\card{I}-1})  \enspace ,
  \end{align*}
  since $p_1\leq 1$.
  Note that for the first equality we used Lemma \ref{lem:shape}, and for the latter equality we used Proposition~\ref{prop:projection}.
\end{proof}

\begin{lemma}\label{lem:SIM-product}
  Let $\{p_k\}_{k\in [n]}$ be a submodular and symmetric price schedule.
  Then the region $D_{[k]}$ is congruent to the product of SIM-bodies
  \[
    \Lambda\Big(1-(p_k - p_{k-1}), \dots, 1-(p_2 -p_1), 1-p_1\Big) \times
    \Lambda\Big(p_{k+1} - p_k, \dots, p_n - p_{n-1}\Big) \enspace ,
  \]
  for all $k \in [n]$.
\end{lemma}
\begin{proof}
  Recall from Lemma \ref{lem:product} that for a submodular price schedule the region $D_I$ can be written as
  \[
    D_I = \pi_I(D_I) \times \pi_{I^c}(D_I) \enspace ,
  \]
  where $\pi_I$ is the projection onto the coordinates in $I$ and $I^c = [n] \setminus I$.
  Lemma 6.4 in \cite{GK+duality:2018} shows, that for a region $D_{[k]}$ of SJA we get
  \[
    \pi_I(D_{[k]}) \ = \ \biggSetOf{x \in [0,1]^k}{\sum_{j\in J} x_j \geq p_{k} - p_{k - \card{J}} \text{ for all } J \subseteq [k]}
  \]
  and
  \[
    \pi_{I^c}(D_{[k]}) \ = \ \biggSetOf{x \in [0,1]^{n-k}}{\sum_{j \in J} x_j \leq p_{k + \card{J}} - p_{k} \text{ for all } J \subseteq [n-k]} \enspace .
  \]
  The proof of those equations only uses the fact, that the price schedule of SJA is submodular and symmetric, thus it is applicable in our more general sense.
  Next, it is immediate, that $\pi_{I^c}(D_{[k]}) = \Lambda(p_{k+1} - p_k, \dots, p_n - p_{n-1})$. 
  The central symmetry around the point $\frac{1}{2} \cdot \1$, where $\1$ is the all ones vector, shows that $\pi_I(D_{[k]})$ is congruent to $\Lambda(1-(p_k - p_{k-1}), \dots, 1-(p_2 -p_1), 1-p_1)$.
\end{proof}

In Section \ref{sec:computing}, we will further discuss the computation of the SJA-prices.
It is interesting to note, that for $n>4$ items, the SJA-prices are not strictly increasing and that they are critical with respect to a selling candidate $\cI_{k,n}$ for some $k<n-1$; see Table~\ref{tab:prices}.
That is, gaps do occur.

For the remainder of this section, we want to compare SJA to other submodular and symmetric auctions and strengthen the result of Proposition \ref{prop:SJA-solvable}.
A direct conclusion of Theorem \ref{thm:uniqueCritPrice} says that any submodular and critical price schedule for $\cI_{k,n}$ is automatically symmetric.
The next lemma complements this statement.

\begin{lemma}\label{lem:gap}
  Let $p = \{p_k\}_{k\in[n]}$ be a submodular and symmetric price schedule with $p_1 < 1$.
  Then there exists $k \in \{0,\dots, n-1\}$ such that $p$ admits $\cI_{k,n}$.
\end{lemma}
\begin{proof}
  For the sake of contradiction, let us assume there is a number $\ell < n$ such that $\vol D_{[\ell]} > 0$ but $\vol D_{[\ell - 1]} = 0$.
  Since $\vol D_{[\ell]} > 0$, we can apply Lemma \ref{lem:relevantPrice} and conclude that $p_{\ell} < p_{\ell + 1}$.
  Next, we want to show that $p_{\ell} \leq p_{\ell -1}$.
  Let $x \in [0,1]^n$ with $x_i = 1$ for $i \in [\ell -1]$ and $x_i = 0$ otherwise.
  Since $\vol D_{[\ell - 1]} = 0$, there is some $\ell'$ such that
  \[
    \sum_{i \in [\ell-1]} x_i - p_{\ell - 1} \ \leq \ \sum_{i \in [\ell']} x_i - p_{\ell'}  \enspace .
  \]
  If $\ell' < \ell-1$ we get $p_{\ell-1} \leq p_{\ell'} + (\ell-1-\ell')p_1$ by submodularity, which results in
  \[
    \sum_{i \in [\ell -1]} x_i - p_{\ell -1} \ = \ \ell-1 - p_{\ell -1} \ \geq \ \ell' - p_{\ell'} + (\ell-1-\ell')(1-p_1) \ > \ \ell' - p_{\ell'} \ = \ \sum_{i \in [\ell']} x_i - p_{\ell'} \enspace .
  \]
  Therefore $\ell' > \ell-1$. 
  But then $\sum_{i \in [\ell']} x_i = \ell - 1 = \sum_{i \in [\ell-1]} x_i$, whence $p_{\ell'} \leq p_{\ell-1}$.
  Applying Lemma \ref{lem:relevantPrice} again, from $\ell' > \ell-1$ we obtain $p_{\ell} \leq p_{\ell-1}$.
  Finally, the submodularity of $p$ implies $p_{\ell+1} \leq p_{\ell} + p_{\ell} - p_{\ell-1} \leq p_{\ell}$, and this contradicts $p_{\ell} < p_{\ell + 1}$.
     
  Further, we need to show that $\vol D_{[n]} > 0$.
     We do this by showing that the point $x = (p_1, \dots, p_1)$ is in $D_{[n]}$.
     This is enough because also $x' \in D_{[n]}$ if $x_i' \geq x_i$ for all $i \in [n]$.
     By submodularity we get $p_n \leq p_{\ell} + (n-\ell)\cdot p_1$ for all $\ell < n$.
     Therefore
     \[
       \sum_{i \in [n]} x_i - p_n \ = \ n \cdot p_1 - p_n \ \geq \ \ell \cdot p_1 - p_{\ell} + (n-\ell)(p_1 - p_1) \ \geq \ \sum_{i \in [\ell]} x_i - p_{\ell} \enspace ,
     \]
     which proves that $x \in D_{[n]}$.
\end{proof}

We call a submodular auction \emph{degenerate} if at least one of the submodularity conditions is satisfied with equality, i.e., there are bundles which are tight.
By Proposition \ref{prop:SJA-solvable} SJA is a candidate for the optimal auction.
The next results shows, that it is the only candidate among all submodular and symmetric auctions which are nondegenerate.

\begin{theorem}\label{thm:SJA-optimal}
  Let $M$ be a deterministic, submodular and symmetric auction for $n$ items which is not degenerate.
  If $M$ satisfies the necessary critical point conditions (1)--(4) of Proposition \ref{prop:necessary}, then it is the SJA.
\end{theorem}
\begin{proof}
  Let $(p_1, \dots, p_n)$ be the vector of SJA-prices for $n$ items and let $k \in \{1,\dots,n-1\}$ be chosen such that $\cI_{k,n}$ is the selling candidate of SJA.
  Let $(q_1, \dots, q_n)$ be the price schedule of $M$. 
  Assume that $q \neq p$.
  We will show, that $q$ then infringes one of the critical conditions.
  If $q_1 = 1$, then $\vol_{D_{[1]}} = 0$ and by submodularity, $q_i > 1$ for $i > 1$. 
  Therefore $\delta(\{1\}) < 0$ and thus condition \eqref{cond:notTightTrivVol} is not satisfied.
  Hence we assume $q_1 < 1$.
  Now we can apply Lemma \ref{lem:gap} and find some $k' \in \{0,\dots,n-1\}$ such that $q$ admits $\cI_{k',n}$.
  Since $q\neq p$, we know $k' < k$.
  For $q$ to satisfy condition \eqref{cond:notTightPosVol}, it needs to be critical with respect to $\cI_{k',n}$.
  Therefore, the prices $q_i$ need to be defined such that $\delta_n({[i]}) = 0$ for all $i \leq k'$.
  Since SJA is critical with respect to $\cI_{k,n}$, this is also true for $p_i$.
  As in the proof of Theorem \ref{thm:uniqueCritPrice}, there is at most one solution to $\delta_n({[i]}) = 0$ for all $i \leq k'$.
  Hence $q_i = p_i$ for $i \leq k'$.
  
  Next, we claim $q_n > p_{k'+1}$:
  For any two integers $1 \leq a \leq b \leq n$, we define the polytope
  \[
    P_a^b \ = \ \biggSetOf{x \in [0,1]^n}{\sum_{i \in I} x_i \leq p_{\ell} \text{ for all } \ell \in \{a, \dots, b\} \text{ and } I \in \binom{[n]}{\ell}} \enspace .
  \]
  Moreover, we define the halfspace $H_r = \smallSetOf{x \in \RR^n}{\sum_{i \in [n]} x_i \leq r}$.
  Let us assume for now that $k' >0$.
  With this notation and using the proof of Theorem \ref{thm:uniqueCritPrice}, as well as the fact that $q_i = p_i$ for $i \leq k'$, $q_n$ is the unique solution to
  \begin{equation}\label{eq:q_n}
    \vol_n \left (P_1^{k'} \cap H_{q_n} \right) = \frac{1}{n+1} \enspace .
  \end{equation}
  Similarly, $p_n$ is the unique solution to $\vol_n (P_1^{k'} \cap P_{k'+1}^k \cap H_{p_n}) = \frac{1}{n+1}$.
  By Lemma \ref{lem:relevantPrice}, we get $p_{k'} < p_{k'+1} < \dots < p_k < p_n$.
  Therefore, assuming $q_n \leq p_{k'+1}$ implies $P_1^{k'} \cap P_{k'+1}^k \cap H_{q_n} = P_1^{k'} \cap H_{q_n}$ and thus $p_n = q_n$.
  But this is a contradiction to $p_n > p_k$.
  Note that for $k' = 0$, we can not use the proof of Theorem \ref{thm:uniqueCritPrice} to guarantee that there is at most one solution to \ref{eq:q_n}.
  But since $q_1<1$ and $\vol D_{\{1\}} = 0$ we have $q_n < 1$, which is enough to show that \ref{eq:q_n} has at most one solution, namely $q_n = \tfrac{n!}{n+1}$.
  
  Finally, we show that $q$ does not satisfy condition \eqref{cond:notTightTrivVol}.
  Since for $q$ we have $\vol_n D_{k'+1} = 0$, we can assume $q_{k'+1} = q_n$.
  Hence $q_{k'+1} > p_{k'+1}$. 
  But since $p_{k'+1}$ is chosen such that $\vol_n (P_1^{k'} \cap H_{p_{k'+1}}) = 1 - \frac{k'+1}{n+1}$, we get that $\vol_n (P_1^{k'} \cap H_{q_{k'+1}}) > 1 - \frac{k'+1}{n+1}$, which implies $\delta_n([k'+1]) < 0$ for $q$.
\end{proof}

Note that SJA can be degenerate.
This is the case if the gap is at least $2$; in that case the condition $p_n + p_{n-2} \leq p_{n-1} + p_{n-1}$ is tight.
Hence our computations show, that SJA is nondegenerate for $n \leq 7$ and $n=12$, but degenerate for $8 \leq n \leq 11$; see Table~\ref{tab:prices}.
Still SJA is a strong candidate for the optimal auction.
Indeed, it is optimal if no other degenerate auction satisfies the necessary critical point conditions (1)--(4) of Proposition \ref{prop:necessary}.

\begin{corollary}\label{cor:SJA-optimal}
  SJA is optimal among all deterministic, submodular and symmetric auctions, unless other critical degenerate auctions exist.
\end{corollary}
\begin{proof}
  The set of all deterministic, submodular and symmetric auctions is given by the set of submodular and symmetric price schedules.
  This set is bounded and closed, hence it is compact.
  Since the expected revenue is a continuous function depending on the price schedule, there is an optimal price schedule which maximizes the expected revenue.
  If there is no critical price schedule other than SJA, that must be optimal.
\end{proof}

In general, neither symmetry nor submodularity are forced by optimality.
This is what the following example shows.
However, in contrast to our general assumption above, the example employs a distribution for the valuations of the buyer which is not uniform.
\begin{example}[{Babaioff et al.\ \cite[Example 5.2]{BabaioffEtAl:2018}}]
  Let us consider a price schedule for $n=3$ items, where the valuations are drawn independently from the discrete distribution $f$ given by:
  \[
    \begin{tabular*}{.5\linewidth}{@{\extracolsep{\fill}}cccccc@{}}
      \toprule
      $v$ & 0 & 1 & 2 & 5 & 6 \\
      \midrule
      $f(v)$ & 0.1 & 0.1 & 0.4 & 0.1 & 0.3 \\
      \bottomrule
    \end{tabular*}
  \]
  Then one optimal price schedule is defined by $p_{\{1\}} = p_{\{2\}} = p_{\{3\}} = 6$, $p_{\{1,2\}} = p_{\{1,3\}} = 7$, $p_{\{2,3\}} = 8$ and $p_{\{1,2,3\}} = 9$. 
  This price schedule is neither symmetric nor submodular, but has a better expected revenue than any symmetric or submodular price schedule.
\end{example}

\section{Volume Formulae}
\label{sec:volume}
In Section~\ref{sec:auctions} we have proven that critical price schedules are characterized by the volumes of polytopes that appear as the regions of a tropical hypersurface, and that the projections of these polytopes are SIM-bodies.
The latter have been recognized as generalized permutahedra in Section~\ref{sec:sim}.
This opens up several paths for computing and analyzing volumes of SIM-bodies, all of which have their merits.
Giannakopoulos and Koutsoupias proposed to recursively evaluate integrals \cite[(12)]{GK+duality:2018}.
Throughout we fix parameters $\alpha_1 \geq \dots \geq \alpha_n \geq 0$ and abbreviate $\Lambda=\Lambda(\alpha_1,\dots,\alpha_n)$.

\subsection*{Dragon marriage}
A substantial part of Postnikov's work \cite{Postnikov:2009} deals with rewriting generalized permutahedra in order to obtain various formulae for volumes and counting lattice points.
Here we follow the approach in \cite[\S9]{Postnikov:2009}.
With this, for $I\in\tbinom{[n+1]}{k}$ and $k\geq 1$, we set
\begin{equation}\label{eq:y_I}
  y_I \ = \
  \begin{cases}
    \sum_{i=0}^{k - 2} (-1)^{i+k} \binom{k - 2}{i} \alpha_{n-i}, & \text{ if } n+1 \in I \\
    0, & \text{ otherwise}\enspace .
  \end{cases}
\end{equation}
Note that the value $y_I$ only depends on the cardinality of $I$ and the parameters $\alpha_i$.
We write standard simplices as $\Delta_I = \conv\smallSetOf{e_i}{i \in I}$.
The following is a description of the SIM-bodies in terms of \enquote{nested sets} in the sense of \cite[\S7]{Postnikov:2009}.
\begin{proposition} \label{prop:sum}
  For $\alpha_1 \geq \dots \geq \alpha_n \geq 0$ we have 
  \begin{equation}\label{eq:sum}
    \Lambda(\alpha_1,\dots,\alpha_n) \ = \ P_{n+1}\{z_I\} \ = \  \sum_{I \subseteq [n+1]} y_I \Delta_I \enspace,
  \end{equation}
  where $y_I$ is defined as in \eqref{eq:y_I}, and $z_I$ as in \eqref{eq:z_I}.
\end{proposition}
\begin{proof}
  Verify that $\sum_{J \subseteq I} y_J=z_I$ for $\emptyset \neq I \subseteq [n+1]$.
  With Proposition~\ref{prop:projection} the claim follows from \cite[Remark 6.4]{Postnikov:2009}.
\end{proof}
\begin{remark}
  Note that the decomposition \eqref{eq:sum} is not necessarily a Minkowski sum of standard simplices, as the parameters $y_I$ may be negative.
  For instance, with $n = 3$ and $(\alpha_1,\alpha_2,\alpha_3) = (4,3,1)$, we obtain $y_{[4]} = \alpha_3 - 2 \alpha_2 + \alpha_1 = -1$.
\end{remark}

Let $\Phi_n = \smallSetOf{ d \in [n]^n }{ d_1 \leq \dots \leq d_n }$ be the set of weakly ascending integer vectors of length $n$, whose entries lie in the set $[n]$.
For $d\in \Phi_n$ we let $M_{d}$ denote the number of labeled bipartite graphs with left degree vector $(d_1, \dots, d_n)$, sorted in weakly ascending order, which satisfy Hall's condition; i.e., they include a perfect matching.

\begin{corollary}\label{cor:volume}
  For $\alpha_1 \geq \dots \geq \alpha_n \geq 0$ the volume of $\Lambda=\Lambda(\alpha_1,\dots,\alpha_n)$ equals
  \[
    \vol_n ( \Lambda ) \ = \ \frac{1}{n!} \sum_{d \in \Phi_n} M_{d} \ \omega_{d_1} \cdots \omega_{d_n} \enspace  ,
  \]
  where
  \[
    \omega_k \ = \ \sum_{i=0}^{k-1} (-1)^{i+k-1} \binom{k-1}{i} \alpha_{n-i} \enspace.
  \]
  In particular, the normalized volume $n!\cdot\vol_n(\Lambda)$ is an integral polynomial which is homogeneous of degree $n$ in the parameters.
\end{corollary}

\begin{proof}
  It suffices to consider the generic case $\alpha_1 > \dots > \alpha_n > 0$, from which the degenerate cases follow by continuity.
  By Proposition~\ref{prop:sum} and \cite[Corollary 9.4]{Postnikov:2009} we have
  \begin{equation}\label{eq:volume-dragon}
    \vol_n ( \Lambda ) \ = \ \frac{1}{n!} \sum_{(J_1,...,J_n)} y_{J_1} \cdots y_{J_n} \enspace ,
  \end{equation}
  where the sum is taken over ordered collections of nonempty subsets $J_1,\dots,J_n \subseteq [n+1]$ such that, for pairwise distinct $i_1,\dots,i_k$, we have
  \begin{equation}\label{eq:dragon}
    \card \bigl( J_{i_1} \cup \dots \cup J_{i_k} \bigr) \ \geq \ k+1 \enspace .
  \end{equation}
  The volume formula \eqref{eq:volume-dragon} can be simplified by removing zero terms arising from the case distinction in \eqref{eq:y_I}; moreover, $y_I=y_{I'}$ whenever $I$ and $I'$ share the same cardinality.
  We have $\omega_k=y_I$ if $I$ contains $n+1$ and exactly $k$ additional elements in $[n]$.

  For $d_i=\card{J_i}-1$ the product $\omega_{d_1} \cdots \omega_{d_n}$ equals $y_{J_1} \cdots y_{J_n}$, if each $J_i \subseteq [n+1]$ satisfies $n+1\in J_i$.
  From $J_1,\dots,J_n$ with weakly ascending cardinalities we obtain a bipartite graph by taking $J_i\setminus\{n+1\}$ as the set of right neighbors of the left node $i$.
  Now the requirement \eqref{eq:dragon} on $J_1,\dots,J_n$ translates into $\card( (J_{i_1}\setminus\{n+1\}) \cup \dots \cup (J_{i_k}\setminus\{n+1\})  ) \geq k$, which is Hall's condition.
\end{proof}
The inequality \eqref{eq:dragon} is the \emph{dragon marriage condition} \cite[Proposition 5.4]{Postnikov:2009}.
There are many other ways to compute volumes of SIM-bodies: for instance, via integration \cite[Equation~(12)]{GK+duality:2018}.

\begin{example}
  Let us compute the volume of the SIM-body $\Lambda (4,3,1)$ via Corollary \ref{cor:volume}.
  We have $\omega_1 = 1, \omega_2 = 2, \omega_3 = -1$ and thus
  \[
    \begin{aligned}
      \vol_3 ( \Lambda (4,3,1) ) \ &= \ \frac{1}{6} \Big( M_{(1,1,1)} + 2 M_{(1,1,2)} - M_{(1,1,3)} + 4 M_{(1,2,2)} - 2 M_{(1,2,3)} + M_{(1,3,3)} \\
      & \qquad + 8 M_{(2,2,2)} - 4 M_{(2,2,3)} + 2 M_{(2,3,3)} - M_{(3,3,3)} \Big) \\
      &= \ \frac{1}{6} \Big(6 + 2 \cdot 36 - 18 + 4 \cdot 63 - 2 \cdot 54 + 9 + 8 \cdot 24 - 4 \cdot 27 + 2 \cdot 9 - 1 \Big) \\
      &= \ \frac{157}{3} \enspace.
    \end{aligned}
  \]
  The formula \cite[(12)]{GK+duality:2018} gives $v(a,b) := \vol_2(\Lambda(a,b)) = \frac{1}{2}a^2 + ab - \frac{1}{2}b^2$ and
  \[
    \begin{aligned}
      \vol_3 ( \Lambda (a,b,c) ) \ &= \ \int_{0}^{c} v(a,b) dt + \int_{c}^{b} v(a,b+c-t) dt + \int_{b}^{a} v(a+b-t,c) dt \\
      &= \ \frac{1}{6} (c^3 + 3 c^2 (b - 2 a) - 3 c (2 b^2 - 2 b a - a^2) - 2 b^3 + 3 b^2 a + 3 b a^2 + a^3) \enspace .
    \end{aligned}
  \]
  Substituting $a=4$, $b=3$ and $c=1$ recovers $\vol_3 ( \Lambda(4,3,1)) = \frac{157}{3}$.
\end{example}

\subsection*{Lawrence's method}
We can exploit that the regular SIM-bodies are simple polytopes.
Let $P=\smallSetOf{x\in\RR^n}{Ax\leq b}$ be a rational simple $n$-polytope with vertex set $V$ such that $(A \ b)$ is facet defining.
We assume that the rows of $A\in\ZZ^{m\times n}$ are the primitive (outward) facet normals.
Then each vertex $v$ is incident with exactly $n$ facets, i.e., there is an invertible $n{\times}n$-submatrix $A_v$ of $A$ such that $A_vv=b_v$, where $b_v$ is the subvector of $b$ corresponding to the rows of $A_v$.
Let $c\in\ZZ^n$ be a linear objective function which is not constant on any edge of $P$.
That is, $c$ induces an acyclic orientation on the vertex-edge graph of $P$.
We consider the vector $\gamma^v:= \invtranspose{A_v} c$, where $\invtranspose{A_v}={(\transpose{A_v})}^{-1}$ is the inverse transpose.
The vector $c$ is \emph{generic} for $P$ if it is not constant on edges and additionally $\gamma_v\neq 0$ for all vertices $v$ of $P$.
For such data Lawrence \cite{Lawrence:1991} showed
\begin{equation}\label{eq:lawrence}
  \vol_n(P) \ = \ \frac{1}{n!} \cdot \sum_{v\in V} \frac{\langle c,v\rangle^n}{|\det A_v| \prod_{i=1}^n \gamma_i^v} \enspace .
\end{equation}
See also \cite[\S3]{BuelerEngeFukuda:2000}.

A rational polytope $P$ is called \emph{Delzant} if the primitive edge directions at each vertex form a lattice basis.
This property is equivalent to the condition that the projective toric variety associated with the normal fan of $P$ does not have any singularities, i.e., it is smooth in the sense of differential geometry; see \cite[Theorem 2.4.3]{toric+varieties}.
Since a lattice basis obviously needs to form a basis, a smooth polytope is necessarily simple, but the converse does not hold, in general.
We will show that the SIM-bodies have that property.
Note that Lawrence formula is particularly useful for Delzant polytopes since then $|\det A_v|=1$ in \eqref{eq:lawrence}, which simplifies the computation; this will be used in Section~\ref{sec:computing}.
We define
\[
A_k = \left[ \begin{array}{c|c}
     C_k & 0 \\ \hline
     0 & -I_{n-k}
\end{array}\right],
\quad C_k = \left[ \begin{array}{cccc}
     1 & & & 0 \\
     1 & 1 & & \\
     \vdots & & \ddots & \\
     1 & 1 & \dots  & 1
\end{array}\right], \quad
b_{\alpha^{(k)}} = \left(\begin{array}{c}
\alpha_1\\
\alpha_1 + \alpha_{2} \\
\vdots \\
\alpha_1 + \dots + \alpha_{k} \\
0 \\
\vdots \\
0
\end{array}\right),
\]
where $A_k\in\RR^{n\times n}$, $C_k \in \RR^{k \times k}$, and $I_{n-k}$ is the identity matrix in $\RR^{(n-k)\times(n-k)}$.
By Proposition~\ref{prop:SIM-vertices} each vertex $v$ of $\Lambda$ can be written as $\sigma(\alpha^{(k)})$, where $\alpha^{(k)} = (\alpha_1, \dots, \alpha_{k}, 0, \dots, 0)$, and $\sigma \in \Sym(n)$ is a permutation.
It follows from Proposition \ref{prop:projection} that $A_k \alpha^{(k)} = b_{\alpha^{(k)}}$.
Note that the matrix $A_k$ only depends on $k$, not on the parameters $\alpha_i$.
\begin{proposition}\label{prop:SIM-smooth}
  Any proper SIM-body is a Delzant polytope.
\end{proposition}
\begin{proof}
  This follows from $\det(A_k)=\pm 1$.
\end{proof}
Now \eqref{eq:lawrence} specializes as follows.
\begin{lemma}\label{lem:SIM-lawrence}
  For $\alpha_1 \geq \dots \geq \alpha_n \geq 0$ the volume of $\Lambda=\Lambda(\alpha_1,\dots,\alpha_n)$ equals
  \begin{equation}\label{eq:SIM-lawrence}
    \vol_n(\Lambda) \ = \ \frac{1}{n!} \cdot \sum_{k \in [n]} \sum_{\sigma \in \Sym(n)} \frac{{\langle \sigma(c), \alpha^{(k)} \rangle}^n}{(n-k)!\prod_{i=1}^n (\invtranspose{A_k} \sigma(c))_i} \enspace,
  \end{equation}
  for any generic objective function $c$.
\end{lemma}
\begin{proof}
  Again it suffices to consider the generic case $\alpha_1 > \dots > \alpha_n > 0$, from which the degenerate cases follow by continuity.
  Let $P_{\sigma}$ be the permutation matrix corresponding to $\sigma$, i.e., $v=P_{\sigma}\alpha^{(k)}$.
  Then we obtain $A_v = A_k P_{\sigma}^{-1}$ and $b_v = b_{\alpha^{(k)}}$.
  In \eqref{eq:SIM-lawrence} we can leave out the origin, which is a vertex, since it does not contribute to the final volume.
  The factor $(n-k)!$ in the denominator comes from the fact, that all permutations which permute the $0$-entries of $\alpha^{(k)}$ count the same vertex.
  Moreover, the scalar product in the numerator satisfies $\langle c, P_{\sigma} \alpha^{(k)} \rangle=\langle P_{\sigma}^{-1} c, \alpha^{(k)} \rangle$, and $P_\sigma^{-1}c = \sigma^{-1}(c)$.
  Finally, notice that iterating over all permutations $\sigma$ or their inverses is the same.
\end{proof}
This leads us to identifying those linear objective functions which are generic.
\begin{lemma}\label{lem:c-generic}
  An objective function $c$ is generic for $\Lambda$ if and only if the coefficients of $c$ are nonzero and pairwise distinct.
  In particular, $c=(1,2,\dots,n)$ is generic.
\end{lemma}
\begin{proof}
  The edge directions of $\Lambda$ are $e_i-e_j$ for $i\neq j$ or $e_i$.
  Consequently, $c$ is not constant on any edge of $\Lambda$ if and only if $c_i\neq c_j$ for $i\neq j$ and $c_i\neq 0$.
  In that case, a routine computation yields $(\invtranspose{A_k} \sigma(c)) _i \neq 0$ for all $i$, $k$ and $\sigma$.
\end{proof}
While the volume formula \eqref{eq:SIM-lawrence} is fairly compact, it can be rewritten to improve the efficiency of evaluation.
Our goal is to get rid of the factor $(n-k)!$ in the denominator by enumerating each vertex only once.
To this end we gather some useful observations.
First, the objective function $c=(1,2,\dots,n)$ does not depend on the parameters $\alpha_i$.
Second, we compute
\[
\invtranspose{A_k} = \left[ \begin{array}{c|c}
     \invtranspose{C_k} & 0 \\ \hline
     0 & -I_{n-k}
\end{array}\right] \,, \qquad
\invtranspose{C_k} = \left[ \begin{array}{ccccc}
     1 & -1 & & & 0 \\
      & 1 & -1 & & \\
      & & \ddots & \ddots & \\
      & & & 1 & -1 \\
     0 & & & & 1
\end{array}\right] \enspace .
\]
We fix some $k \in [n]$.
Recall from the proof of Lemma~\ref{lem:SIM-lawrence} permuting the vector $\alpha^{(k)}$ gives the same as permuting the objective function $c$ by the inverse permutation.
Then two permutations $\sigma_1, \sigma_2$ correspond to the same vertex if and only if the first $k$ coordinates of $\sigma_1^{-1}(c)$ and $\sigma_2^{-1}(c)$ are the same.
This allows us to first choose $k$ entries of $c=(1,2,\dots,n)$ and then go through all $k!$ permutations of those vectors.
We arrive at the following.
\begin{proposition}\label{prop:SIM-lawrence}
  For $\alpha_1 \geq \dots \geq \alpha_n \geq 0$ the volume of $\Lambda=\Lambda(\alpha_1,\dots,\alpha_n)$ equals
  \[
    \vol_n(\Lambda) \ = \ \frac{1}{n!} \sum_{k \in [n]} \sum_{L \in \binom{[n]}{k}} \sum_{\sigma \in \Sym(k)} \frac{{\langle \sigma(L), \alpha^{[k]} \rangle}^n}{\sigma(L)_k \cdot \prod_{i\in [k-1]} \left( \sigma(L)_{i} - \sigma(L)_{i+1} \right) \cdot  \prod_{i \notin L} (-i)} \enspace,
  \]
  where $\alpha^{[k]} = (\alpha_1, \dots, \alpha_k)$, and we sort $L \in \binom{[n]}{k}$ ascendingly to obtain a vector in $\RR^k$.
\end{proposition}

\subsection*{Lorentzian polynomials}
In this section we will exhibit yet another representation of the SIM-bodies which is similar to \eqref{eq:sum}, but here we obtain a proper Minkowski sum decomposition.
This will reveal interesting structural properties of the volume polynomials.
To this and we set $\beta_k := \alpha_k - \alpha_{k+1}$ and $\alpha_{n+1} := 0$.
As the parameters $\alpha_i$ are strictly decreasing, the new parameters $\beta_i$ are strictly positive.
\begin{proposition}\label{prop:SIM_Minkowski}
  We have
  \begin{equation}\label{eq:lorentzian}
    \Lambda(\alpha_1, \dots, \alpha_n) \ = \ \sum_{i = 1}^n \beta_i \, Q(i,n) \enspace ,
  \end{equation}
  where
  \[
    Q(k,n) \ := \ \BiggSetOf{x \in [0,1]^n}{\sum_{i = 1}^n x_i \leq k} \enspace .
  \]
\end{proposition}
\begin{proof}
  We will first check the inclusion from right to left.
  Let $x = \sum_{k=1}^{n} q^{(k)}$ with $q^{(k)} \in \beta_k \, Q(k,n)$.
  For $j,k \in [n]$ we have $q^{(k)}_j \leq \beta_k$ and $\sum_{i=1}^{n} q^{(k)}_i \leq k \beta_k$.
  Thus each set $J \subseteq [n]$ satisfies $\sum_{j \in J} q^{(k)}_j \leq \beta_k \cdot \min\{k, \card{J} \}$.
  We get
  \[
    \begin{aligned}
      \sum_{j \in J} x_j \ = \ \sum_{j \in J} \sum_{k=1}^{n} q^{(k)}_j \ &\leq \ \sum_{k=1}^{n} \beta_k \cdot \min\{k, \card{J} \} \\
      & = \ \sum_{k=1}^{n} (\alpha_{k} - \alpha_{k+1}) \cdot \min\{k, \card{J} \} \ = \ \sum_{k=1}^{\card{J}} \alpha_{k} \enspace ,
    \end{aligned}
  \]
  and thus $x \in \Lambda$.
  
  For the reverse inclusion consider an arbitrary point $x \in \Lambda$.
  This can be written as a convex combination of the vertices.
  As the origin is a vertex of $\Lambda$, we obtain a function $\gamma : [n] \times S_n \rightarrow [0,1]$ with $\sum_{k=1}^n \sum_{\sigma \in S_n} \gamma(k,\sigma) \leq 1$  such that
  \[
    x \ = \ \sum_{k=1}^n \sum_{\sigma \in S_n} \gamma(k,\sigma) \sigma \left( \sum_{i=1}^k \alpha_i e_i\right) \enspace .
  \]
  Note that $\gamma$ is not unique, since $x$ can be written as a convex combination in more than one way.
  Permutations are linear, and so we get:
  \[
    \begin{aligned}
    x \ &= \ \sum_{\sigma \in S_n}  \sum_{k=1}^n \sum_{i=1}^k \gamma(k,\sigma) \alpha_i \sigma(e_i)
        \ = \ \sum_{\sigma \in S_n} \sum_{k=1}^n \sum_{i=1}^k \gamma(k,\sigma) \left( \sum_{j=i}^{n} \alpha_j - \alpha_{j+1}  \right) \sigma(e_i) \\
      &	= \ \sum_{\sigma \in S_n} \sum_{j=1}^n ( \alpha_j - \alpha_{j+1} )  \sum_{i=1}^j \sum_{k=i}^n \gamma(k,\sigma) \sigma(e_i) \\ 
      & = \ \sum_{j=1}^n \beta_j \left[ \sum_{\sigma \in S_n} \sum_{k=1}^j \gamma(k, \sigma) \sigma \left( \sum_{i=1}^k e_i \right) + \sum_{\sigma \in S_n} \sum_{k=j+1}^n \gamma(k, \sigma) \sigma \left( \sum_{i=1}^j e_i \right) \right] \enspace .
    \end{aligned}
    \]
  Since $\sum_{k=1}^n \sum_{\sigma \in S_n} \gamma(k,\sigma) \leq 1$, the vector enclosed in brackets $[...]$ is contained in the polytope $\conv \{ \bigcup_{k=1}^j \cP (\sum_{i=1}^k e_i ) \cup \{\0 \} \} = Q(j, n)$, which finishes the proof.
\end{proof}
Following Brändén and Huh \cite{Lorentzian} a homogeneous real polynomial with positive coefficients is \emph{Lorentzian} if all its iterated partial derivatives all the way down to degree two yield quadratic forms of Lorentzian signature $({-}{+}{+}\cdots{+})$.
It turns out that the volume polynomials of the SIM-bodies belong to this class, up to a linear substitution.
\begin{theorem}
  The polynomial
  \[
    \lambda(\beta_1, \dots, \beta_n) \ = \ \vol_n\Big( \Lambda\left(\,\beta_1+\dots+\beta_n,\, \beta_2+\dots+\beta_n,\, \dots,\, \beta_n\,\right) \Big)
  \]
  is Lorentzian.
\end{theorem}
\begin{proof}
  Recall that $\alpha_1\geq\dots\geq\alpha_n\geq 0$ if and only if $\beta_1,\dots,\beta_n\geq 0$. 
  Proposition \ref{prop:SIM_Minkowski} shows that $\Lambda\left(\sum_{i=1}^n \beta_i, \sum_{i=2}^n \beta_i, \dots, \beta_n\right) = \sum_{i=1}^n \beta_i Q(i,n)$, and the claim follows from \cite[Theorem 4.1]{Lorentzian}.
\end{proof}

\begin{example}
  The normalized volume of a 3-dimensional SIM-body reads
  \[
    \begin{aligned}
      \vol_3(&\Lambda(\alpha_1,\alpha_2,\alpha_3)) \ = \\
      &\alpha_1^3 + 3 \alpha_1^2 \alpha_2 + 3 \alpha_1^2 \alpha_3 + 3 \alpha_1 \alpha_2^2 + 6 \alpha_1 \alpha_2 \alpha_3 - 6 \alpha_1 \alpha_3^2  - 2 \alpha_2^3 - 6 \alpha_2^2 \alpha_3  + 3 \alpha_2 \alpha_3^2 + \alpha_3^3 \enspace .
    \end{aligned}
  \]
  Substituting $\beta_3=\alpha_3$, $\beta_2=\alpha_2-\beta_3$, $\beta_1 = \alpha_1 - \beta_2 - \beta_3$ results in
  \[
    \begin{aligned}
      \lambda(&\beta_1, \beta_2, \beta_3) \ = \\
      &\beta_1^3 + 6 \beta_1^2 \beta_2 + 9 \beta_1^2 \beta_3 + 12 \beta_1 \beta_2^2 + 36 \beta_1 \beta_2 \beta_3 + 18 \beta_1 \beta_3^2 + 5 \beta_2^3 + 18 \beta_2^2 \beta_3 + 18 \beta_2 \beta_3^2 + 6 \beta_3^3  \enspace .
    \end{aligned}
  \]
\end{example}

\section{Computing the SJA-Prices}
\label{sec:computing}

Giannakopoulos and Koutsoupias listed the SJA-prices for up to six items \cite[p.~146]{GK+duality:2018}.
Here we extend that computation to $n\leq 12$ items using \polymake \cite{DMV:polymake} and \HC \cite{HC}, and we also compute the resulting revenues.
While some of our computations employ numerical methods, our results for $n\leq 10$ are fully certified and thus correct, within their margins of error, which are also given explicitly.
We employ the volume computations for SIM-bodies as laid out in Section~\ref{sec:volume}.

Consider the SJA of $n$ items, and $p_k^{(n)}$ is the price at which any bundle of cardinality $k$ is sold, where $k\in[n]$.
When $n$ is clear from the context, we abbreviate $p_k=p_k^{(n)}$.
By Theorem \ref{thm:uniqueCritPrice} and Lemma \ref{lem:DEmptyisSIM}, the SJA-prices satisfy the critical conditions
\begin{equation}\label{eq:SJA-polynomials}
  1 - \frac{k}{n+1} \ = \ \vol_k \Lambda(p_1, p_2 - p_1, \dots, p_k - p_{k-1}) \quad \text{for } k\in[n] \enspace ,
\end{equation}
and this is a system of $n$ rational polynomial equations in the $n$ prices.
We abbreviate the volume polynomial as $v_k:=\vol_k \Lambda(\alpha_1,\dots,\alpha_n)$.
Notice that the price polynomial $v_k(p_1, p_2 - p_1, \dots, p_k - p_{k-1})$ does not depend on $n$; yet the constant term $1 - \frac{k}{n+1}$ does.
\begin{example}\label{exmp:SJA-polynomials}
  For $n=3$ the polynomial equations \eqref{eq:SJA-polynomials} read
  \begin{align*}
    1-\frac{1}{4} \ =& \  p_1 \\
    1-\frac{2}{4} \ =& \ - p_1^2 + 2 p_1 p_2 - \frac{1}{2} p_2^2 \\
    1-\frac{3}{4} \ =& \ \frac{1}{2} p_1^3 - \frac{3}{2} p_1^2 p_3 - 3 p_1 p_2^2 + 6 p_1p_2p_3 - \frac{3}{2} p_1 p_3^2 + p_2^3 - \frac{3}{2} p_2^2 p_3 + \frac{1}{6} p_3^3 \enspace .
  \end{align*}
\end{example}

The first step is to compute the (coefficients) of the volume polynomials $v_k$.
Here we followed two approaches.
First, we computed the volume polynomial of a proper SIM-body in \polymake via Proposition~\ref{prop:SIM-lawrence}.
Then we applied the linear substitution as in \eqref{eq:SJA-polynomials} to obtain the price polynomials.
Interestingly, because of cancellation effects caused by the substitution, the support of the price polynomials seems to be considerably smaller than the support of the volume polynomials.
Therefore, we also computed the price polynomials directly, by exploiting \cite[Equation~(12)]{GK+duality:2018} and implementing the integration of polynomials in a way tailored to SJA.
That second method is superior for large $n$.

The SJA-prices arise as solutions to the system \eqref{eq:SJA-polynomials} of price polynomials, but they are constrained by the submodularity condition \eqref{eq:submodular}.
Here it reads
\begin{equation}\label{eq:SJA-submodular}
  p_{k} - p_{k-1} \ \leq \ p_{k-1} - p_{k-2} \quad \text{for } k\in[n] \enspace,
\end{equation}
and we use the convention $p_0=0$ and $p_{-1}=-1$.
It follows from Theorem~\ref{thm:uniqueCritPrice} that the polynomial system \eqref{eq:SJA-polynomials} has at most one real solution which satisfies \eqref{eq:SJA-submodular}; see also \cite[Lemma 6.2]{GK+duality:2018}.
Note that \eqref{eq:SJA-polynomials} is valid only for weakly increasing prices, i.e., $p_1\leq p_2 \leq \dots \leq p_n$.
However if during computation we get a price $p_{k} < p_{k-1}$, we can accommodate for it by normalizing the prices on the fly and setting $p_{k-1} = p_k$.
This way, equation \eqref{eq:SJA-polynomials} reads $v_k(p_1,p_2-p_1, \dots, p_k - p_{k-2},0) = 1-\tfrac{k}{n+1}$.
This step may be necessary multiple times.
Combining \eqref{eq:SJA-polynomials} with \eqref{eq:SJA-submodular} then yields the Algorithm \ref{algo:SJA-prices} for computing the SJA-prices; its correctness follows from Theorem~\ref{thm:uniqueCritPrice}.

\begin{algorithm}[th]
  \dontprintsemicolon
  \Input{integer $n\geq 1$}
  \Output{SJA-prices $p= (p_1, \dots, p_n)$ or $\emptyset$, if $\cI_{n-1,n}$ is not solvable}
  \For{$k \in [n]$}{
    $i \leftarrow 1$ \;
    \While{$i < k$}{
      $w(x)$ $\leftarrow$ $v_k(p_1, p_2 - p_1 , \dots, p_{k-i} - p_{k-i-1}, x - p_{k-i}, 0, \dots, 0)$ \;
      $\xi$ $\leftarrow$ solution to $w(x) = 1 - \frac{k}{n+1}$ with $ x - p_{k-1} \leq p_{k-1} - p_{k-2}$ \;
      \lIf{$\xi$ does not exist}{\Return $\emptyset$}
      \If{$\xi\geq p_{k-1}$}{$p_k$ $\leftarrow$ $\xi$ \\ \Break}
      $i$ $\leftarrow$ $i+1$ \;
    }
    \For{$j \in [i-1]$}{
      $p_{k-j}$ $\leftarrow$ $p_k$ \;
    }
  }
  \Return $p= (p_1, \dots, p_n)$ \;

  \caption{Finding the SJA price schedule}
  \label{algo:SJA-prices}
\end{algorithm}


The Algorithm \ref{algo:SJA-prices} cannot be applied naively.
It rather requires an extra layer of considerations which we discuss now.
The nontrivial steps of the procedure is finding the real zeros of a polynomial with real coefficients, which is, of course, classical.
Yet Algorithm \ref{algo:SJA-prices} pretends that those computations are exact.
The multivariate volume polynomial of the SIM-bodies and its linear substitution, the volume polynomial $v_k$, have rational coefficients; so there is no issue with an exact representation.
Yet the univariate polynomial $w(x)=v_k(p_1, p_2 - p_1 , \dots, p_{k-i} - p_{k-i-1}, x - p_{k-i}, 0, \dots, 0)$ depends on the prices $p_1,\dots,p_{k-1}$, which are algebraic numbers.
A naive approach would be to apply numerical methods to find the first price, substitute and iterate the process.
However, the increasing degrees of the price polynomials amplify even a small initial error enormously.
Therefore, instead of this iterative approach we applied \HC to the entire polynomial system \eqref{eq:SJA-polynomials}, as this can return exact certificates \cite{BreidingRoseTimme:2011.05000}.
In this way we were able to compute numerical approximations with at least 15 exact decimal digits (precision at 53 bits) for all $n\leq 10$.
For $n=11$ the \HC computation died for lack of main memory, with more than 130~GB required.
We could apply the (nonexact) iterative method for $n=11$ and $n=12$.
Our results are compiled in Table~\ref{tab:prices}.


Finally, the expected revenue \eqref{eq:revenue} of the SJA reads
\begin{equation}\label{eq:SJA-revenue}
  \Rev(p) \ = \ \sum_{k\in [n]} p_k \cdot \binom{n}{k} \cdot \vol_n\left(D_{[k]}\right) \enspace .
\end{equation}
To compute the volume of the region $D_{[k]}$ recall from Lemma \ref{lem:SIM-product} that this is a product of SIM-bodies.
The computation of \eqref{eq:SJA-revenue} for $n \leq 12$ can be found in Table \ref{tab:revenues-exact}.

The computations were made on an Intel(R) Xeon(R) Processor E5-2630 v4 @ 2.20 GHz 
with openSUSE Leap 15.2, Linux 5.3.18-lp152.36 in about 30 minutes and are certified until $n=10$.
The memory requirement did not exceed 36~GB.

\begin{table}[th]
  \caption{SJA-prices for $n \leq 12$ items, where $p_k^{(n)}$ indicates the price of selling a bundle of $k$ items in an $n$ item setting, and \enquote{--} says that the corresponding bundle is not sold.
    Values for $n\leq 10$ certified by $\HC$ \cite{BreidingRoseTimme:2011.05000}.}
\label{tab:prices}
\renewcommand{\arraystretch}{0.9}
\small
\begin{tabular*}{\linewidth}{@{\extracolsep{\fill}}ccccccccccccc@{}}\toprule
  $n$ & $p^{(n)}_1 $ & $p^{(n)}_2 $ & $p^{(n)}_3 $ & $p^{(n)}_4 $ & $p^{(n)}_5 $ & $p^{(n)}_6 $ & $p^{(n)}_7 $ & $p^{(n)}_8 $ & $p^{(n)}_9 $ & $p^{(n)}_{10} $ & $p^{(n)}_{11} $ & $p^{(n)}_{12}$ \\ 
  \midrule
  $1$	& $1/2$ &&&&&&&&&&&\\
  $2$ 	& $2/3$ & 0.862 &&&&&&&&&&\\
  $3$	& $3/4$ & 1.146 & 1.226 &&&&&&&&&\\
  $4$	& $4/5$ & 1.317 & 1.581 & 1.601 &&&&&&&&\\
  $5$	& $5/6$ & 1.431 & 1.817 & -- & 1.986 &&&&&&&\\
  $6$	& $6/7$ & 1.512 & 1.986 & 2.286 & -- & 2.377 &&&&&&\\
  $7$	& $7/8$ & 1.573 & 2.113 & 2.500 & 2.739 & -- & 2.775 &&&&&\\
  $8$	& $8/9$ & 1.621 & 2.211 & 2.667 & 2.991 & -- & -- & 3.178 &&&&\\
  $9$	& $9/10$ & 1.659 & 2.290 & 2.800 & 3.192 & 3.466 & -- & -- & 3.584 &&&\\
  $10$	& $10/11$ & 1.690 & 2.355 & 2.909 & 3.356 & 3.696 & 3.932 & -- & -- & 3.995 &&\\
  $11$	& $11/12$ & 1.715 & 2.409 & 3.000 & 3.493 & 3.888 & 4.188 & 4.392 & -- & -- & 4.408 &\\
  $12$  & $12/13$ & 1.737 & 2.454 & 3.077 & 3.609 & 4.051 & 4.404 & 4.670 & 4.848 & 4.937 & -- & 5.013 \\
  \bottomrule
\end{tabular*}
\end{table}


\begin{table}[th]
\caption{Expected revenues $\Rev=\Rev(p_1^{(n)}, \dots, p_n^{(n)})$ of the SJA for $n = 1, \dots, 12$.}
\label{tab:revenues-exact}
\renewcommand{\arraystretch}{0.9}
\small
\begin{tabular*}{\linewidth}{@{\extracolsep{\fill}}ccccccccccccc@{}}\toprule
  $n$ & 1 & 2 & 3 & 4 & 5 & 6 & 7 & 8 & 9 & 10 & 11 & 12 \\
  \midrule
  $\Rev$ & 0.25 & 0.549 & 0.875 & 1.220 & 1.576 & 1.943 & 2.318 & 2.699 & 3.086 & 3.478 & 3.883 & 4.336 \\
  \bottomrule
\end{tabular*}
\end{table}

\section{Conclusion}

The revenue optimization of auctions with multiple items is notoriously difficult.
In the present paper this is reflected in metric intricacies concerning special classes of generalized permutahedra and ensuing algebraic questions.
We believe that answers to the following open questions will lead to a better understanding.

As already suggested by Giannakopoulos and Koutsoupias, the SJA seems to be the only immediate candidate for an optimal auction in the submodular and symmetric setting.
A major technical obstacle in any attempt to confirm or refute the corresponding \cite[Conjecture 4.3]{GK+duality:2018} seems to be the following.
\begin{question}
  How large is the gap of SJA for $n$ items?
\end{question}
From Table~\ref{tab:prices} we see that the gap equals two for $n=11$, and it is one for $n=12$.
So as a function of $n$ the gap is not monotone.

There may be more to say from the point of view of tropical geometry.
The support of the tropical polynomial $u$ in \eqref{eq:tropical_polynomial} is formed by the vertices of the unit cube.
The prices, which enter as the coefficients of $u$, give rise to a regular subdivision, $\cS(u)$, on $[0,1]^n$ which is dual to the tropical hypersurface $\cT(u)$; see \cite[\S1.2]{ETC}.
In this way, finding optimal prices becomes a nonlinear optimization problem over the secondary cone of $\cS(u)$.
\begin{question}
  What can be learned from this dual point of view?
  Does it help to understand submodular auctions which are degenerate?
\end{question}


\bibliographystyle{amsplain}
\bibliography{references.bib}

\end{document}